\documentclass[11pt]{article}

\usepackage{ulem}
\usepackage{framed}

\usepackage{amsthm}
\usepackage{amssymb}
\usepackage{amscd}
\usepackage{amsmath}
\usepackage[all]{xy}
\usepackage[top=30truemm,bottom=30truemm,left=25truemm,right=25truemm]{geometry}
\usepackage{comment}
\usepackage{algorithmic,algorithm}
\usepackage{here}
\usepackage{graphicx}
\usepackage{color} %
\usepackage{enumerate} %
\usepackage{listliketab}
\usepackage{multirow}
\usepackage{slashbox}


\makeatletter
  
  \@addtoreset{algorithm}{section}
\makeatother

\theoremstyle{definition}

\theoremstyle{definition}

\theoremstyle{plain}
\newtheorem{theo}{Theorem}

\theoremstyle{plain}
\newtheorem{theor}{Theorem}

\theoremstyle{plain}

\theoremstyle{plain}

\theoremstyle{plain}
\newtheorem{thm}{Theorem}[subsection]

\theoremstyle{definition}

\theoremstyle{definition}

\theoremstyle{definition}

\theoremstyle{definition}

\theoremstyle{definition}
\newtheorem{rem}[thm]{Remark}

\theoremstyle{plain}
\newtheorem{prop}[thm]{Proposition}

\theoremstyle{plain}
\newtheorem{lem}[thm]{Lemma}

\theoremstyle{plain}
\newtheorem{cor}[thm]{Corollary}

\theoremstyle{definition}

\theoremstyle{definition}

\theoremstyle{definition}

\theoremstyle{definition}

\theoremstyle{definition}

\numberwithin{equation}{subsection}

\def\Z{\mathbb{Z}}

\def\F{\mathbb{F}}

\def\Ker{{\rm Ker}}

\newcommand{\bbP}{{\operatorname{\bf P}}}

\newcommand{\Proj}{\operatorname{Proj}}

\newcommand{\Gal}{\operatorname{Gal}}

\newcommand{\GL}{\operatorname{GL}}

\newcommand{\gA}{\operatorname{A}}
\newcommand{\gB}{\operatorname{B}}
\newcommand{\gC}{\operatorname{C}}

\newcommand{\gH}{\operatorname{H}}

\newcommand{\gO}{\operatorname{O}}

\newcommand{\gT}{\operatorname{T}}
\newcommand{\gU}{\operatorname{U}}
\newcommand{\gV}{\operatorname{V}}
\newcommand{\gW}{\operatorname{W}}

\newcommand{\Aut}{\operatorname{Aut}}

\newcommand{\diag}{\operatorname{diag}}

\newcommand{\modulo}{\operatorname{mod}}

\newcommand{\Tr}{\operatorname{Tr}}
\newcommand{\cris}{\operatorname{cris}}

\newcommand{\dR}{\operatorname{dR}}

\makeatletter
\def\@seccntformat#1{\csname the#1\endcsname. }
\makeatother
\makeatletter
\renewcommand\section{\@startsection {section}{1}{\z@}%
 {-3.5ex \@plus -1ex \@minus -.2ex}%
 {2.3ex \@plus.2ex}%
 {\normalfont\large\bfseries}}
\makeatother

\begin{document}

\title
{\bf Enumerating superspecial curves of genus $4$\\ over prime fields}
\author
{Momonari Kudo\thanks{Graduate School of Mathematics, Kyushu University.
E-mail: \texttt{m-kudo@math.kyushu-u.ac.jp}}
\ and Shushi Harashita\thanks{Graduate School of Environment and Information Sciences, Yokohama National University.
E-mail: \texttt{harasita@ynu.ac.jp}}}
\maketitle
\begin{abstract}
In this paper we enumerate nonhyperelliptic
superspecial curves of genus $4$ over prime fields of characteristic $p\le 11$.
Our algorithm works for nonhyperelliptic curves
over an arbitrary finite field in characteristic $p \ge 5$.
We execute the algorithm for prime fields of $p\le 11$ with our implementation on a computer algebra system Magma.
Thanks to the fact that
the cardinality of $\F_{p^a}$-isomorphism classes of
superspecial curves over $\F_{p^a}$ of a fixed genus 
depends only on the parity of $a$, this paper contributes
to the odd-degree case for genus $4$, whereas
\cite{KH16} contributes to the even-degree case.
\end{abstract}

\section{Introduction}

In this paper a curve means a non-singular projective variety of dimension one.
A curve over a perfect field $K$ of characteristic $p>0$ is said to be {\it superspecial} if its Jacobian
is isomorphic to a product of supersingular elliptic curves over
the algebraic closure $\overline K$ of $K$.
This paper aims to enumerate nonhyperelliptic superspecial curves of genus $4$ over prime fields $\F_p$ for $p\le 11$.

This work contributes to the problem on finding or enumerating
maximal or minimal curves over $\F_{p^2}$,
since it is known that any maximal or minimal curve over $\F_{p^2}$
is superspecial.
Conversely any superspecial curve descends to a maximal or minimal curve
over $\F_{p^2}$, see the proof of \cite[Theorem 1.1]{Ekedahl}.

The motivation to study the case over {\it prime} fields comes from the fact that
the enumeration over $\F_p$ and $\F_{p^2}$ is essential
for that over general finite fields.
Indeed, in Proposition \ref{PropositionGeneralFiniteFields}
we shall see the general fact that
the number of $\F_{p^a}$-isomorphism classes of
superspecial curves over $\F_{p^a}$ of fixed genus depends only on the parity of $a$,
see also \cite[Theorem 1.3]{XYY16} by Xue, Yang and Yu
for an analogous result in the case of abelian varieties.

In the literature, there are many works on the enumeration of superspecial curves
over algebraically closed field.
The case of elliptic curves is due to Deuring \cite{Deuring}.
If $g\le 3$, some theoretical approaches
are available,
since any principally polarized abelian variety of dimension $g\le 3$
is the Jacobian variety of a (possibly reducible) curve, see Oort-Ueno \cite{OU}.
In the case of principally polarized abelian varieties,
the number of isomorphism classes of superspecial ones is described by a class number of
a quaternion unitary group, see Ibukiyama-Katsura-Oort \cite[Theorem 2.10]{IKO},
and the explicit formulae of those class numbers are given by Hashimoto-Ibukiyama \cite{HI} for $g=2$
and by Hashimoto \cite{Hashimoto} for $g=3$.
The enumeration of superspecial curves for $g\le 3$ is done by
removing the contribution of reduced curves.
Contrary to this story over algebraically closed field,
such explicit enumerations over finite fields have not been completed yet,
except for $g=1$ case (cf. Xue-Yang-Yu \cite[Prop. 4.4]{XYY16}).
But some results on the existence are known. For example,
%
it is shown that
there exists a maximal curve of genus $g$
over $\F_{p^{2e}}$ if $g=2$ and $p^{2e}\ne 4,9$
(cf. Serre \cite[Th\'eor\`eme 3]{Serre1983})
and if $g=3$, $p\ge 3$ and $e$ is odd
(cf. Ibukiyama \cite[Theorem 1]{Ibukiyama}).
See Ibukiyama-Katura \cite{IK} for the enumeration
of principally polarized abelian varieties over $\overline{\F_p}$
which can descend to those over $\F_p$.

If $g\ge 4$,
any theory working for curves of genus $g$ in arbitrary large characteristic $p$
has not been found.
The case of $g=4$ is a next target;
For $p=5$, Fuhrmann-Garcia-Torres~\cite{FGT} found a maximal curve $C_{0}$ of genus $4$ over $K= \F_{25}$, 
and proved that it gives a unique isomorphism class over $\overline{K}$.
For $p \leq 7$, all superspecial curves of genus $4$ over $\F_{p^2}$ were computationally enumerated in \cite{KH16}.
In particular, the result of~\cite{KH16} enumerated all the maximal curves over $K=\mathbb{F}_{25}$, which are included in the unique isomorphism class of $C_0$ over $\overline{K}$.
The result over $\F_{49}$,
together with results in Serre~\cite{Serre}, Howe~\cite{Howe}
and Howe-Lauter \cite{Howe-Lauter},
determined the exact value of the maximal number $N_{49}(4)$ of
the rational points of curves of genus $4$ over $\F_{49}$,
see \cite[Corollary 5.1.3]{KH16}.
This contributed to the table at {\tt manypoints.org}~\cite{ManyPoints}
about bounds of $N_q(g)$, updated
after the paper \cite{GV} by van der Geer and Vlugt.

There is no superspecial curve of genus $g=4$ over $\F_p$ for $p = 2,3$ by \cite[Theorem 1.1]{Ekedahl},
and for $p=7$ by \cite[Theorem B]{KH16}.
Here are our main theorems:

\begin{theo}\label{MainTheorem}
There exist precisely $7$ superspecial curves of genus $4$ over $\mathbb{F}_{5}$ up to isomorphism over $\mathbb{F}_{5}$.
(Note that there exists precisely $1$ superspecial curve of genus $4$ over $\mathbb{F}_{5}$ up to isomorphism over the algebraic closure,
cf. \cite[Corollary 5.1.1]{KH16}.)
\end{theo}

\begin{theo}\label{MainTheorem2}
There exist precisely $30$ nonhyperelliptic superspecial curves of genus $4$ over $\mathbb{F}_{11}$ up to isomorphism over $\mathbb{F}_{11}$.
Moreover, there exist precisely $9$ nonhyperelliptic superspecial curves of genus $4$ over $\mathbb{F}_{11}$ up to isomorphism over the algebraic closure.
\end{theo}

We also have explicit defining equations of the superspecial curves in Theorems \ref{MainTheorem} and \ref{MainTheorem2} (but omit them in the statement).
Many of them
define maximal curves over $\mathbb{F}_{p^2}$.
For example, we found the following superspecial curve over $\mathbb{F}_{11}$;
Let $Q= 2 x w + 2 y z$, and $P= x^2 y + x^2 z + y^3 + 8 y^2 z + 3 y z^2 + 10 y w^2 + 10 z^3 + 10 z w^2$,
which define one of the $30$ superspecial curves over $\mathbb{F}_{11}$.
Then $C = V(P, Q)$ is a maximal curve over $\mathbb{F}_{11^2}$.
Indeed, the number of its $\mathbb{F}_{11^2}$-rational points is $210$, which coincides with the Hasse-Weil upper bound $q +1 + 2 g \sqrt{q} $ for $q=11^2$.
For the other equations, see Sections \ref{subsec:sscurves} and \ref{subsec:comp_result}, or a table of the web page of the first author~\cite{HPkudo}.

We prove Main Theorem with help of computational results.
The idea of our enumeration method in this paper is based on \cite{KH16}, but an improvement is required:
In \cite[Section 5.2]{KH16}, the authors gave an algorithm (Main Algorithm together with a pseudocode in \cite[Algorithm 5.2.1]{KH16}) to enumerate nonhyperelliptic superspecial curves of genus $4$.
As showed in \cite{KH16}, a nonhyperelliptic curve $C$ of genus $4$ over $K$ is given by an irreducible quadratic form $Q$ and an irreducible cubic form $P$ in $K [x,y,z,w]$.
Regarding coefficients in $P$ as indeterminates, one computes $(P Q)^{p-1}$, and then a multivariate system over $K$ is derived from our criterion for the superspeciality (for details on the criterion, see \cite[Section 3.1]{KH16} or Section 2.1 of this paper).
Considering a tradeoff between a brute-force and Gr\"{o}bner bases techniques, we solve the system with the {\it hybrid method} given in \cite{BFP}.
Here the hybrid method is a method for solving multivariate systems by combining the brute-force on some coefficients with Gr\"{o}bner bases techniques.
For each solution, we test whether $C=V (P, Q)$ is non-singular or not.
In this way, one can enumerate all nonhyperelliptic superspecial curves of genus $4$ over $K$, but an improvement is required to get the result for $q=p=11$ since $p=11$ is not so small.

In this paper, we shall give a modified version of the algorithm in \cite{KH16}.
We here briefly describe the difference between the previous algorithm (Main Algorithm in \cite{KH16})
and the modified version.
Our modification considers optimal coefficients in $P$ to be regarded as indeterminates not only in solving algebraic equations but also in computing the multiplication $(P Q)^{p-1}$.
More concretely,
in the previous version, we first choose and fix the number of the indeterminates in solving multivariate systems.
In other words, we use the {\it same} number of indeterminates in computing $(P Q)^{p-1}$ and solving multivariate systems.
From outputs obtained by the previous algorithm in our experiments,
we observe that the computation of $(P Q)^{p-1}$ might be dominant for large $p$ if each multivariate system is quite efficiently solved.
This depends on the value of $p$, rather than the number of unknown coefficinets in $P$ to be regarded as indeterminates.
From this, we consider increasing the number of the indeterminates in the computation of $(P Q)^{p-1}$, but not changing (or reducing) that
in solving multivariate systems.
In other words, we use {\it different} number of indeterminates in computing $(P Q)^{p-1}$ and solving multivariate systems.
As described above, we consider two kinds of optimal tuples of coefficients in $P$ to be regarded as indeterminates, and doubly use the brute-force on coefficients.
Following the terminology in \cite{BFP},
we call this method {\it double hybrid method} in this paper.
As we will see in this paper,
increasing the number of the indeterminates in the computation of $(P Q)^{p-1}$ allows us to reduce the number of total iterations.
We therefore expect that the modified version with this double hybrid method is extremely faster than the previous version in \cite{KH16} for certain cases.

We also give an algorithm to classify isomorphism classes of superspecial curves of genus $4$, based on the Bruhat decomposition of the orthogonal group associated to the quadratic form $Q$ (cf.\ the algorithm given in \cite{KH16} just computes defining equations, but does not classify isomorphism classes).
With these new algorithms, we completely enumerate the isomorphism classes of superspecial curves of genus $4$ over $\mathbb{F}_5$ and $\mathbb{F}_{11}$.

The automorphism groups of the superspecial curves
 obtained in Theorems A and B,
and the compatibility of this enumeration
and Galois cohomology theory
 will be studied in a separated paper \cite{KHS}.

The structure of this paper is as follows.
In Section 2, we review some basic facts on
nonhyperelliptic curves of genus $4$ and 
a way to compute those Hasse-Witt matrices,
and study the reduction of the enumeration of superspecial curves over an arbitrary finite field to that in the case of degree one or two.
In Section 3 we give a reduction of the defining equations of curves
of genus $4$, refining the way in \cite[Section 4]{KH16}.
In \cite{KH16} we treated only curves with sufficiently many rational points,
but over small fields curves may not have sufficiently many rational points
even if they are maximal.
The reduction in this paper assumes only
that a curve has at least one rational point.
In addition, as $\F_5$ is very small, we need an extra argument over $\F_5$, see Section \ref{ReductionDegenerateCase}.
In Section 4, we state the main results and prove them.
In Appendix we collect the pseudocodes used in Section 4.

\subsection*{Acknowledgments}
This work was supported by
JSPS Grant-in-Aid for Young Scientists (B) 25800008.

\section{Preliminaries}\label{section:2}

We review some basic facts on nonhyperelliptic curves of genus $4$, and a criterion for their superspecialities and non-singularities.

\subsection{Nonhyperelliptic curves of genus $4$ and their superspecialities}\label{CurvesGenus4}
Let $K$ be a perfect field of characteristic $p$, and $C$ a nonhyperelliptic curve of genus $4$ over $K$.
As a canonical curve, $C$ is defined in the $3$-projective space $\bbP^3=\Proj(\overline{K}[x,y,z,w])$ by an irreducible quadratic form $Q$ and an irreducible cubic form $P$ in $x,y,z,w$, see \cite[Chapter IV, Example 5.2.2]{Har}.
As showed in \cite[Section 2.1]{KH16}, we may assume that any coefficient of $Q$ and $P$ belongs to $K$.

It is known that $C$ is superspecial if and only if its Hasse-Witt matrix, which is the matrix of the Frobenius on $H^1 ( C, \mathcal{O}_C)$ for a suitable basis, is zero.
The Hasse-Witt matrix of $C$ is determined by certain coefficients of $(P Q)^{p-1}$, see \cite[Corollary 3.1.6]{KH16} (for more general cases, see \cite[Appendix B]{KH16} or \cite[Section 5]{Kudo}).
Hence we can decide whether $C$ is superspecial or not by computing the coefficients.
We state this fact in Proposition \ref{cor:HW}.

\begin{prop}[\cite{KH16}, Corollary 3.1.6]\label{cor:HW}
With notation as above, $C=V(P,Q)$ is superspecial if and only if all the coefficients of the following monomials in $(P Q)^{p-1}$ are zero:
\begin{equation}
\begin{array}{cccc}
( x^2 y z w )^{p-1}, & x^{2 p-1} y^{p-2} z^{p-1} w^{p-1}, & x^{2 p-1} y^{p-1} z^{p - 2} w^{p -1}, &  x^{2 p -1} y^{p-1} z^{p - 1} w^{p -2}, \\
x^{p-2} y^{2 p-1} z^{p-1} w^{p-1}, & ( x y^2 z w )^{p-1} , & x^{p-1} y^{2 p-1} z^{p - 2} w^{p -1}, &  x^{p -1} y^{2 p-1} z^{p - 1} w^{p -2}, \\
x^{p-2} y^{p-1} z^{2 p - 1} w^{p -1}, & x^{p-1} y^{p-2} z^{2 p-1} w^{p-1}, & ( x y z^2 w )^{p-1} , &  x^{p -1} y^{p-1} z^{2 p - 1} w^{p -2}, \\
 x^{p -2} y^{p-1} z^{p - 1} w^{2 p -1}, & x^{p-1} y^{p-2} z^{p-1} w^{2 p-1}, & x^{p-1} y^{p-1} z^{p - 2} w^{2 p -1}, & ( x y z w^2 )^{p-1}
\end{array} \nonumber
\end{equation}
\end{prop}

\subsection{Non-singularity Testing}\label{subsec:singtest}

Let $K$ be a field and $\overline{K}$ its algebraic closure.
Note that $K$ is not necessarily perfect nor of positive characteristic.
Let $f_1, \ldots , f_t $ be non-constant homogeneous polynomials in $S:=K [ X_0, \ldots , X_r]$.
We denote by $V (f_1, \ldots , f_t)$ the locus in $= \mathrm{Proj} (\overline{K}[X_0,\ldots , X_r])$ of the zeros of $f_1, \ldots , f_t$.
Given $f_1, \ldots , f_t$, we can decide whether $V ( f_1, \ldots , f_t )$ is non-singular or not.
The following is a known fact in computational algebraic geometry.


\begin{lem}[\cite{KH16}, Lemma 3.2.1]\label{lem:non-sing}
With notation as above, let $f_1, \ldots , f_t$ be (non-constant) homogeneous polynomials in $S = K [X_0, \ldots , X_r]$.
We denote by $J (f_1, \ldots , f_t )$ the set of all the minors of degree $r - \mathrm{dim} (V (f_1, \ldots , f_t))$ of the matrix $( \partial f_i / \partial X_j)_{i,j}$.
Then the following are equivalent:
\begin{enumerate}
	\item[{\rm (1)}] The variety $V ( f_1, \ldots , f_t )$ is non-singular.
	\item[{\rm (2)}] For each $0 \leq i \leq r$,
	\[
	1 \in \langle J ( f_1, \ldots , f_t ), f_1, \ldots , f_t, 1 - Y X_i  \rangle_{K [ X_0, \ldots , X_r, Y]},
	\]
	where $Y$ is an extra indeterminate.
\end{enumerate}
\end{lem}
With this criterion, one can test the non-singularity of $V ( f_1, \ldots , f_t )$ by computing a Gr\"{o}bner basis for $\langle J ( f_1, \ldots , f_t ), f_1, \ldots , f_t, 1 - Y X_i  \rangle_{K [ X_0, \ldots , X_r, Y]}$.

\subsection{Enumerating superspecial curves over general finite fields}\label{subsec:GeneralFiniteFields}
Let $K$ be an arbitrary finite field of characteristic $p$.
We reduce the enumeration of $K$-isomorphism classes of superspecial curves over $K$ to that of $\F$-isomorphism classes of superspecial curves over $\F$ for $\F = \F_p$ or $\F_{p^2}$.

\newcommand{\SSp}{{\operatorname{SSp}}}
Let $\SSp_g(K)$ denote the set of $K$-isomorphism classes of superspecial curves over $K$.
The next proposition gives the reduction above.

\begin{prop}\label{PropositionGeneralFiniteFields}
Assume $g\ge 2$.
There exists a bijection
between $\SSp_g(\F_{p^a})$ and $\SSp_g(\F_{p^b})$ if $a\equiv b\ (\modulo 2)$.
\end{prop}

This is an analogue of the result by Xue, Yang and Yu in the case of abelian varieties, see \cite{XYY16}, Theorem 1.3.

\newcommand{\id}{{\operatorname{id}}}
To prove this proposition, we recall a basic fact on the Galois descent theory.
Put $k:=\overline{\F_{p}}$.
Let $\sigma_q$ denote the $q$-th power map on $k$.
Set $\Gamma_q = \Gal(k/\F_q)$.
Note that $\sigma_q$ is a topological generator of $\Gamma_q$.
For a scheme $S$ over $k$,
let $S^{(q)}$ denote $S\otimes_{k,\sigma_q}  k$.
For a morphism $f: S\to T$ of schemes over $k$,
let $f^{(q)}$ denote its base change $S^{(q)} \to T^{(q)}$.
Let $X$ be a quasi-projective variety over $k$.
Assume $|\Aut(X)|<\infty$.
We claim that any isomorphism $\varphi: X^{(q)} \to X$
defines a descent datum.
Let $\varphi_{\sigma_q^i}$ be the isomorphism $X^{(q^i)} \to X$
defined by
$\varphi_{\sigma_q^i} = \varphi\circ\varphi^{(q)}\circ\cdots \circ \varphi^{(q^{i-1})}$.
The cocycle condition $\varphi_{\sigma_q^i\sigma_q^j}
= \varphi_{\sigma_q^i}\circ\varphi_{\sigma_q^j}^{(q^i)}$
is obviously satisfied.
Let $\F_{q^m}$ be a field over which
$X$ and every automorphism of $X$ are defined.
Then $\varphi_{\sigma_q^m} \in \Aut(X)$
and $\varphi_{\sigma_q^m}^{(q^{mj})} = \varphi_{\sigma_q^m}$
for all $j$.
By the assumption $|\Aut(X)|<\infty$,
there exists a natural number $\ell$ such that $(\varphi_{\sigma_q^m})^\ell$
is the identity map $\id_X$ on $X$.
Then we have $\varphi_{\sigma_q^{m\ell}} = \id_X$.
Hence $\{\varphi_{\sigma_q^i}\}$ is a descent datum.
As $X$ is quasi-projective, any descent datum is known to be effective.
In the above setting, for any isomorphism $\varphi: X^{(q)} \simeq X$,
there exist a variety $X_0$ over $\F_q$
and an isomorphism  $\iota: X \to X_0\otimes k$ such that
$\varphi$ factors as $\iota^{-1}\circ\iota^{(q)}: X^{(q)} \to (X_0\otimes k)^{(q)}=X_0\otimes k \to X$.

Now we prove Proposition \ref{PropositionGeneralFiniteFields}.

\begin{proof}[Proof of Proposition \ref{PropositionGeneralFiniteFields}]
Let $C$ be a superspecial curve over $k$ of genus $g$.
Let $\SSp_C(\F_{p^a})$ be the set of
$\F_{p^a}$-isomorphism classes of superspecial curves $C'$ over $\F_{p^a}$
such that $C'\otimes_{\F_{p^a}} k \simeq C$.
It suffices to construct a bijection from $\SSp_C(\F_{p^a})$ and $\SSp_C(\F_{p^{a+2}})$.


It suffices to construct a bijection from the set of descent data of $C$ with respect to $k/\F_{p^a}$
to that with respect to $k/\F_{p^{a+2}}$.
Since $|\Aut(C)|<\infty$, it is enough to give
a bijection from the set of isomorphisms $C^{(p^a)} \to C$
to that of isomorphisms $C^{(p^{a+2})} \to C$.
It is well-kwown that $C$ is defined over $\F_{p^2}$,
see the proof of \cite{Ekedahl}, Theorem 1.1.
Hence there exists an isomorphism  $\varphi_2: C^{(p^2)}\simeq C$.
Let $\varphi_a: C^{(p^a)} \to C$ be an isomorphism.
To $\varphi_a$ we associate an isomorphism
$\varphi_{a+2}: C^{(p^{a+2})} \to C$ by
$\varphi_{a+2} =  \varphi_a\circ\varphi_2^{(p^a)}$.
This clearly gives a desired bijection.
\end{proof}

\section{Reduction of cubic forms}\label{SectionReduction}
Let $p$ be a prime greater than $2$ and $q$ a power of $p$.
Let $\F_q$ be a field consisting of $q$ elements.
We have seen in Section \ref{CurvesGenus4} that
an arbitrary nonhyperelliptic curve of genus $4$ over $\F_q$
is written as $V(P,Q)$ in ${\mathbf P}^3$ where
$P$ is an irreducible cubic form over $\F_q$ and $Q$ is an irreducible quadratic form over $\F_q$.
By the classification theory of quadratic forms, $Q$ is isomorphic to
either of 
{\bf (N1)} $2xw + 2yz$,
{\bf (N2)} $2xw + y^2 - \epsilon z^2$ for $\epsilon\in \F_q^\times \smallsetminus (\F_q^\times)^2$
and {\bf (Dege)} $2yw + z^2$
(cf. \cite[Remark 2.1.1]{KH16}).
Thus we may assume that $Q$ is one of them.
We denote by $\varphi$ the symmetric matrix associated to $Q$.
Let $\gO_\varphi(K)$ and $\tilde \gO_\varphi(K)$
be the orthogonal group
$\{g \in \GL_4(K) \mid {}^t g \varphi g = \varphi\}$
and
the orthogonal similitude group
$\{g \in \GL_4(K) \mid {}^t g \varphi g = \mu \varphi \text{ with } \mu\in K^\times\}$ respectively.
The aim of this section is
to reduce 
the number of indeterminates
in the coefficients in $P$,
considering transformations by elements of $\tilde \gO_\varphi(\F_q)$.
But here we will assume that $V(P,Q)$ has a rational points.
So we start with recalling the fact that
there exists at least one rational point on any superspecial curve over $\F_q$.

\subsection{Existence of rational points on a superspecial curve}
Let $C$ be a curve over a field $K$ of characteristic $p$,
and $J(C)$ its Jacobian variety.
The $p$-rank of $C$ is the rank of the $\Z/p\Z$-module $\Ker(p: J(C)(\overline K) \to J(C)(\overline K)$).
If $C$ is superspecial, then $J(C_{\overline K})$ is a product of supersingular elliptic curves
and therefore its $p$-rank is zero.
It is known that the Frobenius map is nilpotent on $H^1_{\dR}(C)$ if and only if the $p$-rank of $C$ is zero.
The next lemma implies the existence of 
an $\F_q$-rational point on any curve of $p$-rank $0$ over $\F_q$.

\begin{lem}\label{RationalPointOnSSp}
Let $C$ be a curve of $p$-rank $0$ over $\F_q$.
We have
\[
\sharp C(\F_q) \equiv 1 \modulo p.
\]
\end{lem}
\begin{proof}
Write $q = p^a$ and let $W$ be the ring of Witt vectors $W(\F_q)$.
Thanks to the Lefschetz trace formula by Berthelot
\cite{Berthelot}, Chap.~VII, 3.1, Cor. 3.1.11 on p.~581, we have
\[
\sharp C(\F_q) = 1 + q - \Tr(F^a: H^1_{\cris}(C,W)\to H^1_{\cris}(C,W)).
\]
As $H^1_{\cris}(C,W)/pH^1_{\cris}(C,W)=H^1_{\dR}(C)$, 
it suffices to show that the trace of $F^a$ on $H^1_{\dR}(C)$ is zero.
This follows from 
the fact that $F^a$ on $H^1(C,{\mathcal O}_C)$ is
nilpotent if $C$ is of $p$-rank $0$.
\end{proof}

\subsection{The orthogonal groups in the non-degenerate case}\label{non-degenerate case}
The symmetric matrix $\varphi$ of $Q$ in each case of (N1) and (N2) is respectively
\[
\text{\bf (N1)}\quad
\begin{pmatrix}
0 & 0 & 0 & 1\\
0 & 0 & 1 & 0\\
0 & 1 & 0 & 0\\
1 & 0 & 0 & 0
\end{pmatrix},
\qquad \text{\bf (N2)}\quad
\begin{pmatrix}
0 & 0 & 0 & 1\\
0 & 1 & 0 & 0\\
0 & 0 & -\epsilon & 0\\
1 & 0 & 0 & 0
\end{pmatrix},
\]
where $\epsilon\in K^\times\smallsetminus(K^\times)^2$.
Recall the Bruhat decomposition of the orthogonal (similitude) group
\[
\gO_\varphi(K)=\gB\gW\gU \quad\text{and}\quad \tilde\gO_\varphi(K)=\tilde \gB \gW \gU
\]
with $\gB=\gA\gT\gU$ and $\tilde \gB = \gA\tilde \gT \gU$, where $\gA$,$\gT$,$\tilde\gT$, $\gW$ and $\gU$ in each case are given as follows.

\noindent{\bf (N1)} 
Set $\gT = \{\diag(a,b,b^{-1},a^{-1}) \mid a,b\in K^\times\}$ and
$\tilde \gT = \{\diag(a,b,cb^{-1},ca^{-1}) \mid a,b,c\in K^\times\}$,
\[
\gU = \left\{ \left. \begin{pmatrix}1&a&0&0\\0&1&0&0\\0&0&1&-a\\0&0&0&1\end{pmatrix}\begin{pmatrix}1&0&b&0\\0&1&0&-b\\0&0&1&0\\0&0&0&1\end{pmatrix}\right| a,b\in K\right\},\quad \gA = \left\{1_4, 
\begin{pmatrix}
1&0&0&0\\
0&0&1&0\\
0&1&0&0\\
0&0&0&1
\end{pmatrix}\right\}
\]
and $\gW:=\{1_4, s_1, s_2, s_1s_2\}$ with
\[
s_1 = \begin{pmatrix}
0&1&0&0\\
1&0&0&0\\
0&0&0&1\\
0&0&1&0
\end{pmatrix},\quad
s_2 = \begin{pmatrix}
0&0&1&0\\
0&0&0&1\\
1&0&0&0\\
0&1&0&0
\end{pmatrix}.
\]

\noindent{\bf (N2)} Set
$\gA:=\{1_4, \diag(1,1,-1,1)\}$,
\[
\gU=\left\{\left.\begin{pmatrix}
1 & a & 0 & -a^2/2\\
0 & 1 & 0 & -a\\
0 & 0 & 1 & 0\\
0 & 0 & 0 & 1
\end{pmatrix}
\begin{pmatrix}
1 & 0 & b & b^2/(2\epsilon)\\
0 & 1 & 0 & 0\\
0 & 0 & 1 & b/\epsilon\\
0 & 0 & 0 & 1
\end{pmatrix}
\right| a,b \in K
\right\},
\]
\[
\gW:=\left\{1_4, \begin{pmatrix}
0 & 0 & 0 & 1\\
0 & 1 & 0 & 0\\
0 & 0 & -1 & 0\\
1 & 0 & 0 & 0
\end{pmatrix}\right\},\quad
\tilde \gC=\left\{\left.
R(a,b):=\begin{pmatrix}
1 & 0 & 0 & 0\\
0 & a & \epsilon b & 0\\
0 & b & a & 0\\
0 & 0 & 0 & a^2-\epsilon b^2
\end{pmatrix} \right|
\begin{matrix}
a,b\in K,\\
a^2 -\epsilon b^2 \ne 0
\end{matrix}
\right\}.
\]
Put $\gC = \{R(a,b)\in \tilde \gC \mid a^2-\epsilon b^2 = 1\}$
and $\gT=\gH\gC$ and $\tilde \gT= \gH\tilde \gC$,
where $\gH=\{\diag(a,1,1,a^{-1})\mid a \in K^\times\}$.

When we consider the reduction of cubic forms for (N2), we shall use
\begin{lem}[\cite{KH16}, Lemma 4.1.1]\label{RepresentationRotationGroup}
Let $V$ be the vector space consisting of cubics in $y,z$ over $K$.
Consider the natural representation of $\tilde \gC$ on $V$.
\begin{enumerate}
\item[\rm (1)]
The representation $V$ is the direct sum of two subrepresentations $V_1:=\langle y(y^2-\epsilon z^2), z(y^2-\epsilon z^2)\rangle $
and $V_2:=\langle y(y^2+3\epsilon z^2), z(3y^2+\epsilon z^2)\rangle$.
\item[\rm (2)]
$V_1$ consists of four $\tilde \gC$-orbits in $V_1$. They are the orbits
of $\delta y(y^2-\epsilon z^2)$ with $\delta \in \{0\} \cup K^\times/(K^\times)^3$ respectively.
\end{enumerate}
\end{lem}

\subsection{The orthogonal groups in the degenerate case}\label{degenerate case}
The symmetric matrix $\varphi$ for the degenerate case is
\[
\begin{pmatrix}
0 & 0 & 0 & 0\\
0 & 0 & 0 & 1\\
0 & 0 & 1 & 0\\
0 & 1 & 0 & 0
\end{pmatrix}.
\]
As shown in \cite[Lemma 4.2.1]{KH16} we have the Bruhat decomposition:
\[
\gO_\varphi(K) = (\gB \sqcup \gB s \gU) \gV \quad\text{and}\quad 
\tilde \gO_\varphi(K) = (\tilde \gB \sqcup \tilde \gB s \gU) \gV
\]
with $\gB:=\gA \gT \gU$ and $\tilde \gB := \gA \tilde \gT \gU$, where
$\gA:=\{1_4, \diag(1,1,-1,1)\}$,
\[
\gT:=\left\{\left. T(a):=\begin{pmatrix}
1&0&0&0\\
0&a&0&0\\
0&0& 1&0\\
0&0&0&a^{-1}\end{pmatrix} \right| a \in K^\times\right\},\quad
\gU := \left\{\left. U(a):=\begin{pmatrix}
1&0&0&0\\
0&1&a&a^2(2\epsilon)^{-1}\\
0&0&1&a \epsilon^{-1}\\
0&0&0&1\end{pmatrix} \right| a \in K\right\},
\]
\[
s:=\begin{pmatrix}
1&0&0&0\\
0&0&0&1\\
0&0&1&0\\
0&1&0&0\end{pmatrix},\quad
\gV=\left\{\left.\begin{pmatrix}
a&b&c&d\\
0&1&0&0\\
0&0&1&0\\
0&0&0&1\end{pmatrix} \right| a\in K^\times \text{ and } b, c, d\in K\right\}
\]
and $\tilde \gT := \{\diag(1,b,b,b) \mid b\in K^\times\} \gT$.

\subsection{Reduction of cubic forms in the case of (N1)}
Let $K$ be a field of characteristic $p\ne 2$.
Consider the case of $Q=2xw+2yz$.
Let $P$ be an irreducible cubic form in $x,y,z,w$ over $K$.
Assume that $C=V(P,Q)$ has a $K$-rational point.
We use the notation in Section \ref{non-degenerate case} (N1).
\begin{enumerate}
\item[1.] Considering $\modulo Q$, it suffices to consider only $P$ which has no term containing $xw$.
\begin{eqnarray}\label{N1_GeneralFormOfP}
P &= & a_1 x^3 + (a_2y+a_3z)x^2+(a_4y^2 + a_5yz + a_6z^2)x \nonumber\\
&& + a_7y^3 + a_8y^2 z + a_9y z^2 + a_{10}z^3\\
&& + (a_{11}y^2+a_{12}yz+a_{13}z^2)w + (a_{14}y+a_{15}z)w^2 + a_{16}w^3. \nonumber
\end{eqnarray}
\item[2.] 
By the assumption $C(K)\ne\emptyset$ and considering the action of $\gW$,
there is a rational point with non-zero $w$-coordinate.
Let $(-bc, b, c, 1)$ be such a $K$-rational point on $C$,
which provides us an element of $\gO_\varphi(K)$
\[
\begin{pmatrix}
-bc & -b & -c & 1\\
b & 0 & 1 & 0\\
c & 1 & 0 & 0\\
1 & 0 & 0 & 0
\end{pmatrix}.
\]
Let $P'$ be the cubic obtained by transforming $P$ by this element.
One can check that the $x^3$-coefficient of $P'$ is $P(-bc,b,c,1)=0$.
Thus we may assume that
the $x^3$-coefficient $a_1$ of $P$ is zero.
\item[3.]
\begin{enumerate}
\item[$\bullet$] If $a_2\ne 0$ or $a_3 \ne 0$, then
considering $y\leftrightarrow z$, we may assume $a_2 \ne 0$.
Then the transformation of an element of $\gU$ eliminates the $xy^2$-term and the $xyz$-term from $P$.
\item[$\bullet$] The case of $a_2=a_3=0$.
In this case $C$ is singular at $(1,0,0,0)$.
\end{enumerate}
\item[4.] The composition of a certain element
($x\mapsto cx, w\mapsto w/c$, $y\mapsto dy, z\mapsto z/d$) of $\gT$ and
a constant-multiplication to the whole $P$
transforms $P$ into a cubic where
the $x^2y$-coefficient is $1$ and the $x^2z$-coefficient is $0$
or a representative of an element of $K^\times/(K^\times)^2$
and the $xz^2$-coefficient is in $\{0,1\}$.
\end{enumerate}

\begin{lem}\label{NewReductionLemmaN1}
An element of $\tilde\gO_\varphi(K)$ transforms $P$ into
\begin{eqnarray*}
&& (y + b_1 z)x^2  + b_2 xz^2 \\
&&+ a_1 y^3 + a_2 y^2z + a_3 yz^2 +  a_4 z^3 \\
&&  + (a_5 y^2 + a_6 yz + a_7 z^2)w + (a_8y + a_9z)w^2 + a_{10}w^3,
\end{eqnarray*}
for $a_1,\ldots,a_{10}\in K$
and for $b_1\in\{0\}\cup K^\times/(K^\times)^2$ and $b_2\in\{0,1\}$.
\end{lem}

\subsection{Reduction of cubic forms in the case of (N2)}
Let $K$ be a field of characteristic $p\ne 2,3$.
Recall that
the quadratic form in (N2) case is $Q=2xw + y^2 - \epsilon z^2$, where
$\epsilon \not\in(K^\times)^2$.
Consider an irreducible cubic form $P$ in $K[x,y,z,w]$.
Assume that $C=V(P,Q)$ has a $K$-rational point.
We use the notation in Section \ref{non-degenerate case} (N2).
\begin{enumerate}
\item[1.] Considering $\modulo Q$, it suffices to consider only $P$ which has no term containing $xw$, \eqref{N1_GeneralFormOfP}.
\item[2.] 
By the assumption, we have a $K$-rational point $(r,s,t,u)$ on $C$.
If both of $r$ and $u$ were zero,
then $Q(r,s,t,u)=0$ implies $s=t=0$.
Hence $r\ne 0$ or $u\ne 0$.
Considering the action of $\gW$, we may assume $u\ne 0$.
Let $(-(b^2-\epsilon c^2)/2, b, c, 1)$ be such a rational point on $C$,
which provides us an element of $\gO_\varphi(K)$
\[
\begin{pmatrix}
-(b^2-\epsilon c^2)/2 & -b & \epsilon c & 1\\
b & 1 & 0 & 0\\
c & 0 & 1 & 0\\
1 & 0 & 0 & 0
\end{pmatrix}.
\]
Let $P'$ be the cubic obtained by 
transforming $P$ by this element.
The $x^3$-coefficient of $P'$ is $P(-(b^2-\epsilon c^2)/2,b,c,1)=0$.
Thus we may assume that $P$ has $a_1 = 0$.
\item[3.] 
\begin{enumerate}
\item[$\bullet$] If $a_2\ne 0$ or $a_3\ne 0$,
an element of $\gU$ transforms $P$ into
a cubic of which $x^1$-coefficient is a constant-multiplication of 
$(y^2-\epsilon z^2)$, where we used $p\ne 3$.
\item[$\bullet$] If $a_2 = a_3 = 0$,
then $C$ is singular at $(1,0,0,0)$.
\end{enumerate}
\item[4.] The composition of an element of $\tilde \gC$ and a constant-multiplication to the whole $P$ transforms
$P$ into a cubic whose terms only in $y,z$ is of the form
\[
\alpha y(y^2-\epsilon z^2) + \beta y(y^2+3\epsilon z^2) + \gamma z(3y^2+\epsilon z^2)
\]
for $\alpha\in\{0,1\}$ and some $\beta,\gamma \in K$. Here we use Lemma \ref{RepresentationRotationGroup}.

\item[5.] There is an element $(x\mapsto cx, w\mapsto w/c)$ of $\gH$ such that
it transforms $P$ into a cubic whose $z^2w$-term is $0$ or $1$.
\end{enumerate}

Thus we obtain the unconditional version of \cite[Lemma 4.4.1]{KH16}:

\begin{lem}\label{ReductionLemmaN2}
An element of $\tilde\gO_\varphi(K)$ transforms $P$ into the following form
\begin{eqnarray*}
&& (a_1 y + a_2 z)x^2  + a_3 (y^2-\epsilon z^2)x
+ b_1 y(y^2-\epsilon z^2) + a_4 y(y^2+3\epsilon z^2) + a_5 z(3y^2+\epsilon z^2)\\
&&  + (a_6 y^2 + a_7 yz + b_2 z^2)w + (a_8 y + a_9 z)w^2 + a_{10}w^3
\end{eqnarray*}
for some $a_i\in K$ with $(a_1,a_2)\ne (0,0)$
and for $b_1,b_2\in\{0,1\}$.
\end{lem}
\subsection{Degenerate case}\label{ReductionDegenerateCase}
We assume that $p\ne 2,3$.
The case of $q>5$ has been treated in \cite[Section 4.5]{KH16}.
Here we study the case of $q=5$. Assume $K=\F_5$ before the next lemma.
\begin{enumerate}
\item[1.] An element ($x\mapsto x+ay+bz+cw$) of $\gV$
transforms $P$  into a cubic without
terms of $x^2y$, $x^2z$, $x^2w$. We may assume that
the coefficients of $x^2y$, $x^2z$, $x^2w$ of
$P$ are zero. 
\item[2.] Considering $\modulo Q$,
we may assume that there is no term containing $yw$ in $P$,
since $yw\equiv -2^{-1}z^2 \modulo Q$.
\item[3.]
\begin{enumerate}
\item[(I)] If there exists an element of $\gO_\varphi(\F_5)$ stabilizing $x$
which transforms $P$ into $P'$
with non-zero term of $y^3$,
an element of $\gU$ transforms $P'$ into one without term of $y^2z$,
and
the same reduction as steps 4, 5 in \cite[Section 4.5]{KH16}
works. The final reduced form is
as in Lemma \ref{ReductionLemmaDegenerate} (1) below,
which is of the same form as in the case of $q>5$.

\item[(II)] Otherwise $P$ has to be of the form
\begin{equation}\label{3-II}
a_0x^3 + (a_1 y^2 + a_2 z^2 + a_3 w^2 + a_4 yz + a_5 zw)x + a_6(y^2z+zw^2).
\end{equation}
Indeed, we may consider only $P$ whose $y^3$-term and $w^3$-term
are zero, considering the action of $s$ (the transposition of $y$ and $w$).
The general form of $P$ is
\[
a_0x^3 + (a_1 y^2 + a_2 z^2 + a_3 w^2 + a_4 yz + a_5 zw)x
+ a_6 y^2z + a_7 yz^2 + a_8 z^3 + a_9 z^2w + a_{10}zw^2.
\]
The element of $s\gU s$
given by $z\mapsto z-cy$, $w\mapsto w + cz - 2^{-1}c^2y$ for $c\in \F_5$
transforms $P$ into a cubic form,
whose $y^3$-coefficient is
\[
(a_{10}-a_6)c + a_7 c^2 - a_8 c^3 + 2a_9 c^4.
\]
This is zero for every $c\in\F_5$ if and only if
$a_6=a_{10}$ and $a_7=a_8=a_9=0$. 
As $P$ is irreducible, we have $a_6 \ne 0$.
\end{enumerate}
\end{enumerate}
Remaining steps in case (II):
\begin{enumerate}
\item[4.] 
Composing some element ($y\mapsto cy, w\mapsto w/c$) of $\gT$ and some constant-multiplication to the whole $P$, we transform $P$ into a cubic where
$a_6$ in (\ref{3-II}) is $1$
and $a_5$ is $0$ or $a_0^{1/3}$. Here we used $(\F_5^\times)^3 = \F_5^\times$.
\item[5.] The transformation $x\mapsto d\cdot x$ for a certain $d\in K^\times$ sends $P$ to a cubic whose coefficient of $x^3$ is $1$.
Then the coefficient of $xzw$ becomes $0$ or $1$ in case (II).
\end{enumerate}

Summarizing this reduction for $q=5$ and that for $q>5$
obtained in \cite[Lemma 4.5.1]{KH16},
we have the following lemma:

\begin{lem}\label{ReductionLemmaDegenerate}
An element of $\tilde\gO_\varphi(K)$ transforms $P$ into the following form {\rm (1)} if $\sharp K > 5$,
and into either of the following forms {\rm (1)} and {\rm (2)} if $\sharp K = 5$.
\begin{enumerate}
\item[\rm (1)]
\begin{eqnarray*}
&& a_0x^3 + (a_1 y^2 + a_2 z^2 + a_3 w^2 + a_4 yz + a_5 zw)x\\
&& + a_6y^3 + a_7z^3 + a_8 w^3 + a_9yz^2 + b_1 z^2 w + b_2 zw^2,
\end{eqnarray*}
for some $a_i\in K$ with $a_0, a_6\in K^\times$ and for $b_1,b_2\in\{0,1\}$,
where the leading coefficient of $R:=a_1 y^2 + a_2 z^2 + a_3 w^2 + a_4 yz + a_5 zw$ is $1$ or $R=0$;
\item[\rm (2)]
\[
x^3 + (a_1 y^2 + a_2 z^2 + a_3 w^2 + a_4 yz + b_1 zw)x + y^2z + zw^2
\]
for $a_i\in K=\F_5$ and $b_1\in\{0,1\}$.
\end{enumerate}
\end{lem}

\section{Main results}\label{sec:main_results}

In this section, we prove Theorems \ref{MainTheorem} and \ref{MainTheorem2} with help of computational results.
The computational results shall be described in Section \ref{subsec:comp_result}.
We choose and fix a primitive element $\zeta^{(q)}$ of $\F_q$ for each of $q=5$ and $q=11$ throughout this section.


\subsection{Superspecial curves over $\mathbb{F}_{5}$ and $\mathbb{F}_{11}$}\label{subsec:sscurves}

\begin{theor}\label{MainTheorem}
There exist precisely $7$ superspecial curves of genus $4$ over $\mathbb{F}_{5}$ up to isomorphism over $\mathbb{F}_{5}$. 
The seven isomorphism classes are given by $C_i = V(Q, P_i)$ with $Q=2yw + z^2$ and
\begin{eqnarray*}
P_1 & =&  x^3+y^3+w^3,\\
P_2 & =& x^3+2y^3+w^3,\\
P_3 & =& x^3+y^3+w^3+zw^2,\\
P_4 & =& x^3+y^3+2w^3+zw^2,\\
P_5 & =& x^3+y^3+3w^3+zw^2,\\
P_6 & =& x^3+y^3+4w^3+zw^2,\\
P_7 & =& x^3+y^2z+zw^2.
\end{eqnarray*}
$($Note that there exists precisely $1$ superspecial curve of genus $4$ over $\mathbb{F}_{5}$ up to isomorphism over the algebraic closure,
cf. $\mbox{\cite[Corollary 5.1.1]{KH16}}$.$)$
\end{theor}

\begin{proof}
Let $C$ be a curve of genus $4$.
Similarly to the proof of \cite[Theorem A]{KH16}, we may assume that $C$ is nonhyperelliptic, and written as $C=V(P,Q)$ for an irreducible quadratic form $Q$ and an irreducible cubic form $P$ in $\F_{5} [x, y, z, w]$.
We may also assume that $Q$ is either of (N1) $2xw+2yz$, (N2) $2xw+y^2-\epsilon z^2$, or (Dege) $2yw+z^2$, where $\epsilon$ is an element in $\mathbb{F}_5^{\times} \smallsetminus ( \mathbb{F}_5^{\times})^2$.
Moreover it suffices to consider the case (Dege), say $Q=2 y w + z^2$.
%
%
%
%
By Lemma \ref{ReductionLemmaDegenerate}, the cubic form $P$ is assumed to be of the following form:
\begin{enumerate}
\item
\begin{eqnarray*}
&& a_0x^3 + (a_1 y^2 + a_2 z^2 + a_3 w^2 + a_4 yz + a_5 zw)x\\
&& + a_6y^3 + a_7z^3 + a_8 w^3 + a_9yz^2 + b_1 z^2 w + b_2 zw^2
\end{eqnarray*}
for $a_i\in K = \mathbb{F}_5$ and $b_1,b_2\in\{0,1\}$,
where $a_0, a_6\in K^\times = \mathbb{F}_5^{\times}$, or
\item[\rm (2)] 
\[
x^3 + (a_1 y^2 + a_2 z^2 + a_3 w^2 + a_4 yz + b_1 zw)x + y^2z + zw^2
\]
for $a_i\in K=\F_5$ and $b_1\in\{0,1\}$.
\end{enumerate}
It follows from Proposition \ref{prop:Degenerate_q5} in Section \ref{subsec:comp_result} that $C=V (P, Q)$ is superspecial if and only if $P$ is one of $P_i$ for $1 \leq i \leq 7$.
\end{proof}



\begin{theor}\label{MainTheorem2}
There exist precisely $30$ nonhyperelliptic superspecial curves of genus $4$ over $\mathbb{F}_{11}$ up to isomorphism over $\mathbb{F}_{11}$. 
The thirty isomorphism classes are given by $(\mathrm{N1})$ $C_i = V(Q, P_i^{{\rm (N1)}})$ with $Q=2 x w + 2 y z$ for $1 \leq i \leq 8$ as in $\mathrm{Proposition}$ $\mathrm{\ref{prop:N1q11}}$, $(\mathrm{N2})$ $C_i = V(Q, P_i^{\rm (N2)})$ with $Q=2 x w + y^2 - \epsilon z^2$ for $1 \leq i \leq 5$ as in $\mathrm{Proposition}$ $\mathrm{\ref{prop:N2q11}}$, and $(\mathrm{Dege})$ $C_i=V(Q,P_i^{{\rm (Dege)}})$ with $Q= 2 y w + z^2$ as in $\mathrm{Proposition}$ $\ref{prop:Degenerate_q11}$.
Moreover, there exist precisely $9$ nonhyperelliptic superspecial curves of genus $4$ over $\mathbb{F}_{11}$ up to isomorphism over the algebraic closure (see Corollary \ref{MainCorollary}).
\end{theor}
\begin{proof}
Let $C$ be a nonhyperelliptic curve of genus $4$ over $\mathbb{F}_{11}$.
As in the proof of Theorem \ref{MainTheorem} the curve $C$ is written as $C=V(P,Q)$ for an irreducible quadratic form $Q$ and an irreducible cubic form $P$ in $\F_{11} [x, y, z, w]$, where $Q$ is either of (N1) $2xw+2yz$, (N2) $2xw+y^2-\epsilon z^2$ and (Dege) $Q = 2yw+z^2$.
Here $\epsilon$ is an element in $\mathbb{F}_{11}^{\times} \smallsetminus ( \mathbb{F}_{11}^{\times})^2$.
Let $\zeta:=\zeta^{(11)}$ be a generator of the cyclic group $\mathbb{F}_{11}^{\times}$.
We first consider the non-degenerate cases (N1) and (N2).
\begin{description}
	\item[(N1){\rm :}] By Lemma \ref{NewReductionLemmaN1}, the curve $C=V(P,Q)$ is $\F_{11}$-isomorphic to $V(P',Q)$ for some
	\begin{equation}
	\begin{split}
	P' =& ( y + b_1 z) x^2 + b_2 x z^2 \\
	& + a_1 y^3 + a_2 y^2 z + a_3 y z^2 + a_4 z^3 \\
	& + ( a_5 y^2  + a_6 y z  + a_7 z^2 ) w + ( a_8  y + a_9 z ) w^2 + a_{10} w^3,
	\end{split}\nonumber
	\end{equation}
	where $a_1, \ldots , a_{10} \in \mathbb{F}_{11}$, $b_1 \in \{ 0, 1, \zeta \}$ and $b_2 \in \{ 0, 1 \}$.
	By Proposition \ref{prop:N1q11} in Section \ref{subsec:comp_result}, the curve $V (P', Q)$ is superspecial if and only if $P'$ is one of $P_i^{({\rm N1})}$ for $1 \leq i \leq 8$.
	\item[(N2){\rm :}] By Lemma \ref{ReductionLemmaN2}, the curve $C=V(P,Q)$ is $\F_{11}$-isomorphic to $V(P',Q)$ for some
	\begin{equation}
	\begin{split}
	P' = & ( a_1 y + a_2 z ) x^2 + a_3 (y^2 - \epsilon z^2) x + b_1 y ( y^2 - \epsilon z^2 ) + a_4 y ( y^2 + 3 \epsilon z^2 ) + a_5 z (3 y^2 + \epsilon z^2 ) \\
	& + ( a_6 y^2 + a_7 y z + b_2 z^2 ) w + ( a_8 y + a_9 z ) w^2 + a_{10} w^3,
	\end{split}\nonumber
	\end{equation}
	where $(a_1, a_2) \neq (0, 0)$ and $b_1, b_2 \in \{ 0, 1 \}$.
	By Proposition \ref{prop:N2q11} in Section \ref{subsec:comp_result}, the curve $V (P', Q)$ is superspecial if and only if $P'$ is one of $P_i^{({\rm N2})}$ for $1 \leq i \leq 5$.
\end{description}
We next consider the degenerate case (Dege): $Q=2 y w + z^2$.
\begin{description}
	\item[(Dege){\rm :}] 
It follows from Lemma \ref{ReductionLemmaDegenerate} that $C=V(P,Q)$ is $\F_{11}$-isomorphic to $V(P',Q)$ for some
	\begin{equation}
	\begin{split}
	P^{\prime} =& a_0 x^3 + ( a_1 y^2 + a_2 z^2 + a_3 w^2 + a_4 y z  + a_5 z w ) x  \\
	& + a_6 y^3 + a_7 z^3 + a_8 w^3 + a_9 y z^2 + b_1 z^2 w + b_2 z w^2,
	\end{split}\nonumber
	\end{equation}
	where $a_0, a_6 \in \mathbb{F}_{11}^{\times}$ and $b_1, b_2 \in \{ 0, 1 \}$.
	By Proposition \ref{prop:Degenerate_q11} in Section \ref{subsec:comp_result}, the curve $V (P', Q)$ is superspecial if and only if $P'$ is one of $P_i^{({\rm Dege})}$ for $1 \leq i \leq 17$.
\end{description}
Summarizing the above descriptions, we have the theorem.
\end{proof}

\begin{cor}\label{MainCorollary}
Any nonhyperelliptic superspecial curve of genus $4$ over $\mathbb{F}_{11}$ is isomorphic over $\overline{\mathbb{F}_{11}}$ to one of the curves $V(Q^{{\rm (N1)}}, P_i^{({\rm alc})})$ for $1 \leq i \leq 3$, or $V(Q^{({\rm Dege})}, P_j^{({\rm alc})})$ for $4 \leq j \leq 9$, where $Q^{{\rm (N1)}} := 2 x w + 2 y z$, $Q^{({\rm Dege})}:= 2 yw + z^2$ and
\begin{eqnarray}
P_1^{(\mathrm{alc})}&:=& x^2 y + x^2 z + 2 y^2 z + 5 y^2 w + 9 y z^2 + y z w + 4 z^3 + 3 z^2 w + 10 z w^2 + w^3, \nonumber \\
P_2^{(\mathrm{alc})} &:=& x^2 y + x^2 z + y^3 + y^2 z + 7 y z^2 + 4 y w^2 + 2 z^3 + 9 z w^2, \nonumber \\
P_3^{(\mathrm{alc})} &:=& x^2 y + x^2 z + y^3 + 8 y^2 z + 3 y z^2 + 10 y w^2 + 10 z^3 + 10 z w^2, \nonumber \\
P_4^{(\mathrm{alc})} &:=& x^3 + y^3 + w^3, \nonumber \\
P_5^{(\mathrm{alc})} &:=& x^3 + y^3 + z^3 + 5 w^3, \nonumber \\
P_6^{(\mathrm{alc})} &:=& x^3 + x w^2 + y^3, \nonumber \\
P_7^{(\mathrm{alc})} &:=& x^3 + x z w + y^3 + 7 z^3 + w^3, \nonumber \\
P_8^{(\mathrm{alc})} &:=& x^3 + x y z + x w^2 + y^3 + 5 z^3 + 4 w^3, \nonumber \\
P_9^{(\mathrm{alc})} &:=& x^3 + x y z + 6 x w^2 + y^3 + 8 z^3 + 8 w^3. \nonumber 
\end{eqnarray}
\end{cor}
\begin{proof}
The result follows from the proof of  Theorem \ref{MainTheorem2} together with Propositions \ref{prop:N1q11} -- \ref{prop:for_cor}.
\end{proof}


\subsection{Modified version of Main Algorithm in \cite{KH16}}\label{subsec:algorithm}

In \cite[Section 5.2]{KH16}, an algorithm (Main Algorithm together with a pseudocode in \cite[Algorithm 5.2.1]{KH16}) to enumerate superspecial curves of genus $4$ was given.
In this subsection, we improve the algorithm in \cite{KH16}.

Let $C$ be a nonhyperelliptic curve of genus $4$.
As we have seen in Section \ref{CurvesGenus4}, the curve $C$ is defined by an irreducible quadratic form $Q$ and an irreducible cubic form $P$ in $K [x, y, z, w]$, say $C=V(P,Q)$.
The cubic form $P$ can be transformed into
\begin{equation}
\sum_{i=1}^t a_i p_i + \sum_{j=1}^u b_j q_j \label{input_P}
\end{equation}
for some cubics $p_i$'s and $q_j$'s, and some exact elements $a_i$'s and $b_j$'s in $K$.
We would like to enumerate {\it all} $( a_1, \ldots , a_t, b_1, \ldots , b_u )$ such that $C = V( P, Q)$ is superspecial.
In the following, we describe our modified version of the algorithm in \cite{KH16} for the enumeration.

\paragraph{Modified Version of Main Algorithm in \cite{KH16}:}
We denote by $\mathcal{M}$ the set of the $16$ monomials given in Proposition \ref{cor:HW}.
Let $Q$ be a quadratic form over $K:= \mathbb{F}_q$.
Let $p_1, \ldots , p_t$, and $q_1, \ldots , q_u$ be cubics over $K$.
We assume here that $(a_1, \ldots , a_t, b_1, \ldots , b_u)$ can take all elements of a subset $\mathcal{A}$ of $K^{t+u}$.
Our aim is to compute all $(a_1, \ldots , a_t, b_1, \ldots , b_u)$ such that $C = V ( P, Q )$ are superspecial for $P = \sum_{i=1}^ta_i p_i + \sum_{j=1}^u b_j q_j$.
Our enumeration algorithm is divided into the following four steps:

\begin{enumerate}\setcounter{enumi}{0}
	\item[(0)] Choose $1 \leq s_1 \leq t$ and indices $k_1, \ldots , k_{s_1}$, and then regard $a_{k_1}, \ldots , a_{k_{s_1}}$ as indeterminates.
	The remaining part $(a_{k_1^{\prime}},\ldots,a_{k_{t_1}^{\prime}})$  ($\{ k_1^{\prime}, \ldots, k_{t_1}^{\prime} \} = \{1,\ldots,t \}\setminus \{k_1,\ldots,k_{s_1}\}$) runs through a subset $\mathcal{A}_1$ of $\mathbb{F}_q^{\oplus (t-s_1)} = \mathbb{F}_{q}^{\oplus t_1}$, which we determine in each case.
\end{enumerate}
For each $( a_{k_1^{\prime}}, \ldots , a_{k_{t_1}^{\prime}} ) \in \mathcal{A}_1$, proceed the following three steps:
\begin{enumerate}
	\item[(1)] Compute $h:=(P Q)^{p-1}$ over $\mathbb{F}_{q} [a_{k_1}, \ldots, a_{k_{s_1}}]$, where $a_{k_1}, \ldots, a_{k_{s_1}}$ are indeterminates. 
	\item[(2)] Choose $1 \leq s_2 \leq s_1$ and indices $i_1, \ldots , i_{s_2}$ such that $\{ i_1, \ldots , i_{s_2} \} \subset \{ k_1, \ldots , k_{s_1} \}$, and then regard $a_{i_1}, \ldots , a_{i_{s_2}}$ as indeterminates.
	The remaining part $(a_{j_1},\ldots,a_{j_{t_2}})$  ($\{j_1,\ldots,j_{t_2}\} = \{k_1,\ldots, k_{s_1}\}\smallsetminus \{i_1,\ldots,i_{s_2}\}$) runs through a subset $\mathcal{A}_2$ of $\mathbb{F}_q^{\oplus (s_1-s_2)}= \mathbb{F}_{q}^{\oplus t_2}$, which we determine in each case.
	\item[(3)] For each $\left( a_{j_1}, \ldots , a_{j_{t_2}} \right) \in \mathcal{A}_2$, proceed the following three sub-procedures:
	\begin{enumerate}
		\item Put
		\[
		\mathcal{S} := \{ c \in \mathbb{F}_q [ a_{i_1}, \ldots , a_{i_{s_2}} ] : c m \mbox{ is a term of } h \mbox{ for some } m \in \mathcal{M} \}.
		\]
		\item Solve the multivariate system $f ( a_{i_1}, \ldots , a_{i_{s_2}} ) = 0$ for all $f \in \mathcal{S}$ over $\mathbb{F}_q$.
		\item For each solution $( a_{i_1}, \ldots , a_{i_{s_2}})$, substitute it into unknown coefficients in $P$, and decide whether $C = V ( P, Q)$ is non-singular or not.
	\end{enumerate}
\end{enumerate}
In Algorithm \ref{alg:enume} of Appendix \ref{sec:code}, we give a pseudocode to proceed the above four steps.

\begin{rem}
In the above algorithm (Modified Version of Main Algorithm in \cite{KH16}), one can take the following procedures instead of (2) - (3):
\begin{enumerate}
\item[$(2)'$] Choose $1 \leq s_2 \leq s_1$ and indices $i_1, \ldots , i_{s_2}$ such that $\{ i_1, \ldots , i_{s_2} \} \subset \{ k_1, \ldots , k_{s_1} \}$, and let $\{j_1,\ldots,j_{t_2}\}$ be the remaining part, i.e., $\{j_1,\ldots,j_{t_2}\} = \{k_1,\ldots, k_{s_1}\}\smallsetminus \{i_1,\ldots,i_{s_2}\}$.
Let $\mathcal{A}_2 \subset \mathbb{F}_q^{\oplus (s_1-s_2)}= \mathbb{F}_{q}^{\oplus t_2}$ be the set of candidates of $(a_{j_1},\ldots,a_{j_{t_2}})$.
Different from the procedure (2) in Modified Version of Main Algorithm in \cite{KH16}, we keep $(a_{j_1},\ldots,a_{j_{t_2}})$ as indeterminates.
\item[$(3)'$] For each $\left( c_{j_1}, \ldots , c_{j_{t_2}} \right) \in \mathcal{A}_2$, proceed the following three sub-procedures:
\begin{enumerate}
		\item[$(a)'$] Put
		\[
		\mathcal{S} := \{ c \in \mathbb{F}_q [ a_{i_1}, \ldots , a_{i_{s_1}} ] : c m \mbox{ is a term of } h \mbox{ for some } m \in \mathcal{M} \},
		\]
		and
		\[
		\mathcal{S}^{\prime}:= \mathcal{S} \cup \{ a_{j_1} - c_{j_1}, \ldots , a_{j_{t_2}} - c_{j_{t_2}} \}.
		\]
		Note that $a_{j_1}, \ldots , a_{j_{t_2}}$ are indeterminates, whereas $c_{j_1}, \ldots , c_{j_{t_2}}$ are exact elements in $\mathbb{F}_q$.
		\item[$(b)'$] Solve the multivariate system $f ( a_{i_1}, \ldots , a_{i_{s_1}} ) = 0$ for all $f \in \mathcal{S}^{\prime}$ over $\mathbb{F}_q$.
		\item[$(c)'$] For each solution $( a_{i_1}, \ldots , a_{i_{s_1}})$, substitute it into unknown coefficients in $P$, decide whether $C = V ( P, Q)$ is non-singular or not.
	\end{enumerate}
\end{enumerate}
In $(2)'$, we add generators $a_{j_1} - c_{j_1}, \ldots , a_{j_{t_2}} - c_{j_{t_2}}$ into $\mathcal{S}$ instead of substituting elements in $\mathbb{F}_q$ into $a_{j_1}, \ldots , a_{j_{t_2}}$.
These alternative procedures give another improvement of Main Algorithm in \cite{KH16}, which we call Another Improved Algorithm here.
We have conducted the computation to enumerate superspecial curves of genus $4$ over $\mathbb{F}_{11}$ by not only Modified Version of Main Algorithm in \cite{KH16}, but also Another Improved Algorithm.
From the outputs, we observe that there are a time-memory trade-off between these two improvements.
This shall be an interesting phenomenon, but we do not precisely deal with Another Improved Algorithm in this paper.
\end{rem}

\paragraph{Our Modification: Double Hybrid Method}
We describe our modification of the previous algorithm (Main Algorithm in \cite{KH16}), and its effects on total time for our enumeration:
Assume for simplicity that $b_i$'s are fixed and that $a_i$'s can take all elements of $\mathbb{F}_q$.
In the following, we denote by
\begin{description}
\item[$t_{mlt}${\rm :}] average time for computing $(PQ)^{p-1}$, and
\item[$t_{GBslv}${\rm :}] average time for solving multivariate systems,
\end{description}
see also Table Notation in Section \ref{subsubsec:timing}.
Let $t$ be the number of unknown $a_i$'s.
In the previous version, we first choose and fix the number of the indeterminates in the computation of solving multivariate systems.
In other words, we use the {\it same} number of indeterminates in computing $(P Q)^{p-1}$ and solving multivariate systems.
For the fixed number $s$, we run through $t - s$ coefficients in $P$.
For each tuple of coefficients, one computes $(P Q)^{p-1}$ over $\mathbb{F}_q [a_{i_1}, \ldots , a_{i_s}] [x,y,z,w]$, and solve a multivariate system in $\mathbb{F}_q [a_{i_1}, \ldots , a_{i_s}]$.
In this case, the number of total iterations is $q^{t-s}$, and hence required time is roughly estimated as
\[
q^{t-s} (t_{mlt} + t_{GBslv}),
\]
where we suppose that non-singularity testing is negligible.
From outputs obtained by the previous algorithm in our experiments, we observe in our enumeration that the computation of $(P Q)^{p-1}$ might be dominant for large $p$ if each multivariate system is quite efficiently solved.
This depends on the value of $p$, rather than the number of indeterminates of the coefficient ring $\mathbb{F}_q [a_{i_1}, \ldots , a_{i_s}]$.
From this, we consider increasing the number of indeterminates in the computation of $(P Q)^{p-1}$, but not changing (or reducing) that in the multivariate system solving step.
In other words, we may not use the same number of indeterminates in computing $(P Q)^{p-1}$ and solving multivariate systems.
As showed in Modified Version of Main Algorithm in \cite{KH16}, our modified algorithm first chooses respectively the number of indeterminates in computing $(P Q)^{p-1}$ and solving multivariate systems, say $s_1$ and $s_2$.
For the fixed $s_1$, we run through $t - s_1$ coefficients in $P$.
For each tuple of coefficients, one computes $(P Q)^{p-1}$ over $\mathbb{F}_q [a_{i_1}, \ldots , a_{i_{s_1}}] [x,y,z,w]$.
After that, we also run through $s_1 - s_2$ coefficients in $P$, and solve a multivariate system in $\mathbb{F}_q [a_{i_1}, \ldots , a_{i_{s_2}}]$.
In this case, required time is roughly estimated as
\[
q^{t-s_1} \left( t_{mlt} + q^{s_1-s_2}  t_{GBslv} \right),
\]
where we suppose that non-singularity testing is negligible.
Hence, if $p$ (or $q$) is large enough and if $t_{GBslv}$ is negligible compared to $t_{mlt}$, we expect
\[
q^{t-s_1} \left( t_{mlt} + q^{s_1-s_2}  t_{GBslv} \right) \ll q^{t-s} \cdot ( t_{mlt} + t_{GBslv}).
\]
For example, if $q=11$, $t=10$, $s_1=9$, $s_2=s=5$, $t_{mlt}=10$ (seconds) and $t_{GBslv}=0.05$ (seconds), we estimate
\[
q^{t-s} (t_{mlt} + t_{GBslv}) \approx 1618562
\]
whereas
\[
q^{t-s_1} \left( t_{mlt} + q^{s_1-s_2}  t_{GBslv} \right) \approx 8162,
\]
which is about 198 times faster than using the previous version.

We call this method {\it double hybrid method} in our enumeration of superspecial curves of genus $4$.
Thanks to this double hybrid method, we have succeeded in finishing all the computations necessary to show the main theorems with this double hybrid method, see also Section \ref{subsubsec:timing}.
Here we heuristically decided $s_1$ and $s_2$ from experimental computations.
\subsection{Enumerating isomorphism classes}\label{subsec:isom}

Let $K = \mathbb{F}_q$ be the field of order $q$, and $Q$ an irreducible quadratic form in $K [x,y,z,w]$.
Let $\varphi$ be the symmetric matrix associated to $Q$.
Let $C_1 = V (Q, P_1)$ and $C_2 = V (Q, P_2)$ be two curves of genus $4$ over $K$ with irreducible cubic forms $P_1$ and $P_2$ in $K [x, y, z, w]$.
The two curves $C_1$ and $C_2$ are isomorphic over $K$ if and only if there exists $g \in \tilde{\rm O}_{\varphi} (K)$ such that
\[
g \cdot P_1 \equiv \lambda P_2 \mbox{ mod } Q 
\]
for some $\lambda \in K^{\times}$.
With this fact, we write down an algorithm for determining whether two curves of genus $4$ are isomorphic over $K$ or not.
Let us focus on the case of (N1), and give an algorithm only for the case in this paper;
as we will state in Remark \ref{rem:isom}, one can construct algorithms for the cases (N2) and (Dege) in similar ways to (N1).
Given a set $\mathcal{P}$ of irreducible cubic forms in $K [x,y,z,w]$, we also give an algorithm to compute a subset $\mathcal{P}^{\prime} \subset \mathcal{P}$ such that $V (Q, P_1)$ and $V( Q, P_2)$ are not isomorphic over $K$ for all $P_1, P_2 \in \mathcal{P}^{\prime}$ with $P_1 \neq P_2$.


We consider the case (N1), that is, $Q = 2 x w + 2 y z$ with
\[
\varphi = 
 \begin{pmatrix}
0 & 0 & 0 & 1\\
0 & 0 & 1 & 0\\
0 & 1 & 0 & 0\\
1 & 0 & 0 & 0
\end{pmatrix}.
 \]
As in Section \ref{SectionReduction}, put
\[
\gT := \{\diag(a,b,b^{-1},a^{-1}) : a,b\in K^\times\}, \quad
\tilde \gT := \{\diag(a,b,cb^{-1},ca^{-1}) : a,b,c\in K^\times\},
\]
\[
U_1(a) := 
\left. \begin{pmatrix}1&a&0&0\\0&1&0&0\\0&0&1&-a\\0&0&0&1\end{pmatrix} \right.
,\quad 
U_2 (b) := \left. \begin{pmatrix}1&0&b&0\\0&1&0&-b\\0&0&1&0\\0&0&0&1\end{pmatrix} \right.
, \quad
\gU := \{ U_1 (a) U_2(b) : a,b\in K\},
\]
\[
\gA := \left\{1_4, 
\begin{pmatrix}
1&0&0&0\\
0&0&1&0\\
0&1&0&0\\
0&0&0&1
\end{pmatrix}\right\},\quad
s_1 := \begin{pmatrix}
0&1&0&0\\
1&0&0&0\\
0&0&0&1\\
0&0&1&0
\end{pmatrix},\quad
s_2 := \begin{pmatrix}
0&0&1&0\\
0&0&0&1\\
1&0&0&0\\
0&1&0&0
\end{pmatrix}.
\]
Let $\gW:=\{1_4, s_1, s_2, s_1s_2\}$.
Put $\gB := \gA\gT\gU$ and $\tilde \gB := \gA\tilde \gT \gU$.
Recall from Section \ref{SectionReduction} that we have
\[
\gO_\varphi(K)=\gB\gW\gU \quad\text{and}\quad \tilde\gO_\varphi(K)=\tilde \gB \gW \gU.
\]
Given irreducible cubic forms $P_1$ and $P_2$ in $K [x,y,z,w]$, we give an algorithm for testing whether $V (Q, P_1)$ and $V(Q, P_2)$ are isomorphic over $K$ or not.
The correctness of this algorithm is straightforward from its construction. 

\paragraph{Isomorphism Testing Algorithm:}
With notation as above, conduct the following procedures for the inputs $P_1$, $P_2$ and $K=\mathbb{F}_q$:
\begin{description}
	\item[{\rm (1)}] Let $t_1, \ldots, t_7$ and $\lambda$ be indeterminates.
	\item[{\rm (2)}] For each $M_{\mathrm{A}} \in \mathrm{A}$ and $M_{\mathrm{W}} \in \mathrm{W}$, we proceed the following four steps:
	\begin{description}
		\item[{\rm (a)}] Put $M_{\tilde{\mathrm{T}}} := \mathrm{diag} (t_1, t_2, t_3 t_2^{-1}, t_3 t_1^{-1}) \in \tilde{\rm T}$, and compute
		\[
		g: = M_{\mathrm{A}} \cdot M_{\tilde{\mathrm{T}}} \cdot U_1 (t_4) \cdot U_2 (t_5) \cdot M_{\mathrm{W}} \cdot U_1 ( t_6 ) \cdot U_2 (t_7) 
		\]
		\item[{\rm (b)}] Construct a multivariate system from the equation
		\[
		g \cdot P_1 \equiv \lambda P_2 \mbox{ mod } Q 
		\]
		together with
		\[
		t_i^{q-1} = 1 \mbox{ for }1 \leq i \leq 3, \quad t_j^q = t_j \mbox{ for } 4 \leq j \leq 7, \mbox{ and} \quad \lambda^{q-1} = 1.
		\]
		Let $\mathcal{S} \subset K [t_1, t_2, t_3, t_4, t_5, t_6, t_7, \lambda ]$ be the set of defining polynomials for the system.
		\item[{\rm (b)}] Compute the reduced Gr\"{o}bner basis $G$ for $\langle S \rangle \subset K [t_1, t_2, t_3, t_4, t_5, t_6, t_7, \lambda ]$ with respect to some term order.
		For computing a Gr\"{o}bner basis, we use known algorithms, e.g., $F_4$.
		\item[{\rm (c)}] If $\sharp G \neq 1$, return ``\textbf{ISOMORPHIC}'', which means that $V (Q, P_1)$ and $V( Q, P_2)$ are isomorphic over $K=\mathbb{F}_q$.
	\end{description}
If the multivariate systems have no solution over $K=\mathbb{F}_q$, i.e., $\sharp G =1$ for all $M_{\mathrm{A}} \in \mathrm{A}$ and $M_{\mathrm{W}} \in \mathrm{W}$, then return ``\textbf{NOT ISOMORPHIC}''.
In this case $V (Q, P_1)$ and $V( Q, P_2)$ are not isomorphic over $K=\mathbb{F}_q$. 
\end{description}
In Algorithm \ref{alg:isomorphic_tilde} of Appendix \ref{sec:code}, we also give a pseudocode for Isomorphism Testing Algorithm.

Next, we give an algorithm for computing isomorphism classes.
Given a family $\mathcal{P}= (P_i)_{i=1}^t$ of irreducible cubics in $K [x,y,z,w]$, the following algorithm (Collecting Isomorphism Classes Algorithm) computes a subset $\mathcal{P}^{\prime} \subset \mathcal{P}$ such that $P_1$ and $P_2$ are not isomorphic over $K$ for all $P_1, P_2 \in \mathcal{P}^{\prime}$ with $P_1 \neq P_2$.

\paragraph{Collecting Isomorphism Classes Algorithm:}
With notation as above, conduct the following procedures for the inputs $\mathcal{P}= (P_i)_{i=1}^t$ and $K=\mathbb{F}_q$:
\begin{description}
	\item[{\rm (1)}] Put $\mathcal{P}^{\prime}:= \emptyset$, and let $FlagList1$ be a sequence of $t$ zeros, say $FlagList1 := ( 0 )_{i=1}^t$.
	\item[{\rm (2)}] For $i=1$ to $t-1$, we proceed the following two steps if $FlagList1[i] = 0$:
	\begin{description}
		\item[{\rm (2-1)}] Replace $\mathcal{P}^{\prime}$ by $\mathcal{P}^{\prime} \cup \{ P_i \}$.
		\item[{\rm (2-2)}] For each $i+1 \leq j \leq t$, test whether $P_i$ and $P_j$ are isomorphic over $\mathbb{F}_q$ or not by Isomorphism Testing Algorithm (or its pseudocode Algorithm \ref{alg:isomorphic_tilde}).
		For a technical reason (Remark \ref{rem:isom} (1)), not regard $t_3$ as an indeterminate here, but take $t_3 =1$ in Isomorphism Testing Algorithm.
		If Isomorphism Testing Algorithm outputs ``\textbf{ISOMORPHIC}'', replace by $1$ the $j$-th entry of $FlagList1$, say $FlagList1[j]:=1$.
	\end{description}
	\item[{\rm (3)}] Put $\mathcal{P}^{\prime \prime}:= \emptyset$, and let $FlagList2$ be a sequence of $\sharp \mathcal{P}^{\prime}$ zeros, say $FlagList2 := ( 0 )_{i=1}^{\sharp \mathcal{P}^{\prime}}$.
	\item[{\rm (4)}] For $i=1$ to $\sharp \mathcal{P}^{\prime}-1$, we proceed the following two steps if $FlagList2[i] = 0$:
	\begin{description}
		\item[{\rm (4-1)}] Replace $\mathcal{P}^{\prime \prime}$ by $\mathcal{P}^{\prime \prime} \cup \{ \mathcal{P}^{\prime}[i] \}$.
		\item[{\rm (4-2)}] For each $i+1 \leq j \leq \sharp \mathcal{P}^{\prime}$, test whether $\mathcal{P}^{\prime} [i]$ and $\mathcal{P}^{\prime} [j]$ are isomorphic over $\mathbb{F}_q$ or not by Isomorphism Testing Algorithm (or its pseudocode Algorithm \ref{alg:isomorphic_tilde}).
		If Isomorphism Testing Algorithm outputs ``\textbf{ISOMORPHIC}'', replace by $1$ the $j$-th entry of $FlagList2$, say $FlagList2[j]:=1$.
	\end{description}
	\item[{\rm (5)}] Return $\mathcal{P}^{\prime \prime}$.
	This $\mathcal{P}^{\prime \prime}$ has the property that $P$ and $P^{\prime}$ are not isomorphic over $K$ for all $P, P^{\prime} \in \mathcal{P}^{\prime \prime}$ with $P \neq P^{\prime}$.
\end{description}
In Algorithm \ref{alg:isomorphic_family} of Appendix \ref{sec:code}, we also give a pseudocode for Collecting Isomorphism Classes Algorithm.

\begin{rem}\label{rem:isom}
\begin{enumerate}
\item In Collecting Isomorphism Classes Algorithm, we first reduce the number of candidates of the isomorphism classes.
More concretely, for each $P_i$, we first remove $P_j$ with $j \geq i+1$ such that $V (Q, P_i)$ and $V (Q, P_j)$ are isomorphic over $K$ via some element of $\mathrm{O}_{\varphi} (K)$.
After that, we determine the isomorphism classes by elements of $\tilde{\mathrm{O}}_{\varphi} (K)$.
\item Using the Bruhat decompositions given in Section \ref{SectionReduction}, one can also construct an algorithm for each of (N2) and (Dege) as a variant of that for (N1).
Let us omit to give algorithms for (N2) and (Dege) in this paper.
\end{enumerate}
\end{rem}

\subsection{Computational parts of our proofs of the main theorems}\label{subsec:comp_result}

In this subsection, we state computational results for our proofs of the main theorems.
Our computational results are shown by executing algorithms given in Sections \ref{subsec:algorithm} and \ref{subsec:isom}.
We implemented the algorithms in Magma~\cite{Magma}, \cite{MagmaHP}, a computer algebra system.
Details on the implementation will be described in Section \ref{subsec:imple}.

\subsubsection{Degenerate case for $q=5$}\label{subsubsec:comp_dege_q5}

\begin{prop}\label{prop:Degenerate_q5}
Consider the quadratic form $Q = 2 y w + z^2 \in \mathbb{F}_5 [x,y,z,w]$ and cubic forms $P \in \mathbb{F}_5 [x,y,z,w]$ of the form
\begin{enumerate}
\item[\rm (i)]
\begin{equation}
\begin{split}
P =& a_0 x^3 + ( a_1 y^2 + a_2 z^2 + a_3 w^2 + a_4 y z  + a_5 z w ) x  \\
& + a_6 y^3 + a_7 z^3 + a_8 w^3 + a_9 y z^2 + b_1 z^2 w + b_2 z w^2,
\end{split}\label{eq:Degeq5(1)}
\end{equation}
for $a_0, a_6 \in \mathbb{F}_{5}^{\times}$ and $b_1, b_2 \in \{ 0, 1 \}$, or
\item[\rm (ii)] 
\begin{equation}
P = x^3 + (a_1 y^2 + a_2 z^2 + a_3 w^2 + a_4 yz + b_1 zw)x + y^2z + zw^2 \label{eq:Degeq5(2)}
\end{equation}
for $a_i \in K= \mathbb{F}_5$ and $b_1 \in\{ 0,1  \}$.
\end{enumerate}
Then a cubic form $P$ of the form \eqref{eq:Degeq5(1)} or \eqref{eq:Degeq5(2)} such that $V(P,Q)$ is superspecial is one of
\begin{eqnarray}
P_1 &=&  x^3+y^3+w^3, \nonumber \\
P_2 &=& x^3+y^3+ 2 w^3, \nonumber \\
P_3 &=& x^3+y^3+w^3+zw^2, \nonumber \\
P_4 &=& x^3+y^3+2w^3+zw^2, \nonumber \\
P_5 &=& x^3+y^3+3w^3+zw^2, \nonumber \\
P_6 &=& x^3+y^3+4w^3+zw^2, \nonumber \\
P_7 &=& x^3+y^2z+zw^2 \nonumber 
\end{eqnarray}
up to isomorphism over $\mathbb{F}_{5}$.
\end{prop}

\begin{proof}
\begin{description}
\item[{\rm (i)}]
Put $t = 10$, $u = 2$ and
\begin{eqnarray}
\{ p_1, \ldots , p_t \} & = & \{ x^3, x y^2, x z^2, x w^2, x y z, x z w, y^3, z^3, w^3, y z^2 \}, \nonumber \\
\{ q_1, \ldots , q_u \} & = & \{ z^2 w, z w^2 \}. \nonumber 
\end{eqnarray}
For each $( b_1, b_2 ) \in \{ 0, 1 \}^{\oplus 2}$ and $a_0, a_6 \in \mathbb{F}_{5}^{\times}$, we proceed the following three steps:
\begin{description}
	\item[{\rm (1)}] Compute $h:= (P Q)^{p-1}$ over $\mathbb{F}_{5} [a_1, \ldots, a_5, a_7, a_8, a_9]$, where $a_1, \ldots, a_5, a_7, a_8, a_9$ are indeterminates. 
	\item[{\rm (2)}] We regard the $8$ coefficients $a_1, \ldots , a_5$, $a_7$, $a_8$, $a_9$ as indeterminates.
	For solving multivariate systems over $\mathbb{F}_{5} [ a_1, \ldots , a_5, a_7, a_8, a_9 ]$ in the next step, we adopt the graded reverse lexicographic (grevlex) order with
\[
a_{8} \prec a_7 \prec a_9 \prec a_3 \prec a_5 \prec a_2 \prec a_4, \prec a_1,
\]
	whereas for $\mathbb{F}_{5} [ x, y, z, w]$, the grevlex order with $w \prec z \prec y \prec x$ is adopted.
	\item[{\rm (3)}] We proceed the following three steps:
	\begin{description}
		\item[{\rm (a)}] Let $\mathcal{S}$ be the set of the coefficients of the monomials of $\mathcal{M}$ in $h$, where the set $\mathcal{M}$ consists of the $16$ monomials in Proposition \ref{cor:HW}.
		Note that $\mathcal{S} \subset \mathbb{F}_{5} [a_1, \ldots, a_5, a_7, a_8, a_9 ]$.
		\item[{\rm (b)}] Solve the multivariate system $f ( a_1, a_2, a_3, a_4, a_5, a_7, a_8, a_9) = 0$ for all $f \in \mathcal{S}$ over $\mathbb{F}_{5}$ with known algorithms via the Gr\"{o}bner basis computation.
		\item[{\rm (c)}] For each solution $(a_1, a_2, a_3, a_4, a_5, a_7, a_8, a_9 )$ of the above system,  substitute it into unknown coefficients in $P$, and decide whether $C = V ( P, Q)$ is non-singular or not. 
	\end{description}
\end{description}

\item[{\rm (ii) }]
For each $b_1 \in \{ 0, 1 \}$, we proceed the following three steps:
\begin{description}
	\item[{\rm (1)}] Compute $h:= (P Q)^{p-1}$ over $\mathbb{F}_{5} [a_1, a_2, a_3, a_4]$, where $a_1, a_2, a_3, a_4$ are indeterminates. 
	\item[{\rm (2)}] We regard the $4$ coefficients $a_1$, $a_2$, $a_3$, $a_4$ as indeterminates.
	For computing Gr\"{o}bner bases over $\mathbb{F}_{5} [ a_1, a_2, a_3, a_4 ]$, we adopt the grevlex order with
\[
a_{3} \prec a_2 \prec a_4 \prec a_1,
\]
	whereas for $\mathbb{F}_{5} [ x, y, z, w]$, the grevlex order with $w \prec z \prec y \prec x$ is adopted.
	\item[{\rm (3)}] As in the case (i), we enumerate $(a_1, a_2, a_3, a_4)$ such that $C$ is superspecial.
\end{description}
\end{description}
Let $\mathcal{P}$ be the list of cubics $P$ such that $V(P,Q)$ was determined to be superspecial in the above procedures (i) and (ii).
For the inputs $\mathcal{P}$ and $q=5$, we execute a variant of Algorithm \ref{alg:isomorphic_family}.
By the outputs of our computation,
a cubic form $P$ of the form \eqref{eq:Degeq5(1)} or \eqref{eq:Degeq5(2)} such that $V(P,Q)$ is superspecial is one of 
$P_i$ for $1 \leq i \leq 7$, up to isomorphism over $\mathbb{F}_{5}$.
\end{proof}

\subsubsection{Case of (N1) for $q = 11$}\label{subsubsec:comp_N1_q11}

\begin{prop}\label{prop:N1q11}
Consider the quadratic form $Q = 2 x w + 2 y z \in \mathbb{F}_{11}[x,y,z,w]$ and cubic forms $P \in \mathbb{F}_{11}[x,y,z,w]$ of the form
\begin{equation}
\begin{split}
P =& ( y + b_1 z) x^2 + b_2 x z^2 \\
& + a_1 y^3 + a_2 y^2 z + a_3 y z^2 + a_4 z^3 \\
& + ( a_5 y^2  + a_6 y z  + a_7 z^2 ) w + ( a_8  y + a_9 z ) w^2 + a_{10} w^3,
\end{split}\label{eq:N1q11}
\end{equation}
where $a_1, \ldots , a_{10} \in \mathbb{F}_{11}$, $b_1 \in \{ 0, 1, \zeta^{(11)} \}$ and $b_2 \in \{ 0, 1 \}$.
Here $\zeta^{(11)}$ is a primitive element of $\mathbb{F}_{11}$.
Then a cubic form $P$ of the form \eqref{eq:N1q11} such that $V(P,Q)$ is superspecial is one of
\begin{eqnarray}
P_1^{({\rm N1})}&= & x^2 y + x^2 z + 2 y^2 z + 5 y^2 w + 9 y z^2 + y z w + 4 z^3 + 3 z^2 w + 10 z w^2 + w^3, \nonumber \\
P_2^{({\rm N1})}&= & x^2 y + x^2 z + y^3 + y^2 z + 7 y z^2 + 4 y w^2 + 2 z^3 + 9 z w^2, \nonumber \\
P_3^{({\rm N1})}&= & x^2 y + x^2 z + y^3 + 8 y^2 z + 3 y z^2 + 10 y w^2 + 10 z^3 + 10 z w^2, \nonumber \\
P_4^{({\rm N1})}&= & x^2 y + x^2 z + y^3 + 9 y^2 z + 2 y^2 w + 3 y z^2 + 3 y z w + 4 y w^2 + 10 z^3 + 2 z^2 w + 6 z w^2, \nonumber \\
P_5^{({\rm N1})}&= & x^2 y + x^2 z + x z^2 + 10 y^2 w + 9 y z^2 + 9 y w^2 + 8 z^3 + 8 z^2 w + 8 z w^2 + 3 w^3, \nonumber \\
P_6^{({\rm N1})}&= & x^2 y + x^2 z + x z^2 + 9 y^2 z + 5 y^2 w + y z w + 8 y w^2 + 3 z^3 + 9 z^2 w + 2 z w^2 + 5 w^3, \nonumber \\
P_7^{({\rm N1})}&= & x^2 y + x^2 z + x z^2 + 4 y^3 + 2 y^2 z + 10 y^2 w + 3 y z^2 + 8 y z w + 8 y w^2 + 8 z^3 + 7 z^2 w + 7 z w^2 + 4 w^3, \nonumber \\
P_8^{({\rm N1})}&= & x^2 y + x^2 z + x z^2 + 9 y^3 + 6 y^2 z + 5 y^2 w + 8 y z^2 + 5 y z w + 2 y w^2 + z^3 + 2 z^2 w + 7 z w^2 + w^3 \nonumber
\end{eqnarray}
up to isomorphism over $\mathbb{F}_{11}$.
\end{prop}

\begin{proof}
Put $t = 10$, $u = 2$ and
\begin{eqnarray}
\{ p_1, \ldots , p_t \} & = & \{ y^3, y^2 z, y z^2, z^3, y^2 w, y z w, z^2 w, y w^2, z w^2, w^3 \}, \nonumber \\
\{ q_1, \ldots , q_u \} & = & \{ x^2 z, x z^2\}. \nonumber 
\end{eqnarray}
We divide our computation into the following three cases (this is our technical strategy to avoid the out of memory errors). 
\begin{description}
\item[{\bf (i) Case of $b_1 \neq 0$}.]
For each $b_1 \in \{ 1, \zeta^{(11)} \}$ and $b_2 \in \{ 0, 1 \}$, execute Modified Version of Main Algorithm in \cite{KH16}, given in Section \ref{subsec:algorithm}.
We here give an outline of our computation together with our choices of $s_1$, $s_2$, $\{ k_1, \ldots , k_{s_1} \}$, $\{ i_1, \ldots , i_{s_2} \}$, $\mathcal{A}_1$, $\mathcal{A}_2$ and a term ordering in the algorithm.
\begin{description}
\item[{\rm (0)}]
We set $s_1:= 8$, and $( k_1, \ldots , k_{s_1} ):= ( 3, \ldots , 10)$ (we regard the $8$ coefficients $a_3$, $a_4$, $a_5$, $a_6$, $a_7$, $a_8$, $a_9$, $a_{10}$ as indeterminates).
Let $\mathcal{A}_1:=\mathbb{F}_{11} \times \mathbb{F}_{11}$.
 \end{description}
For each $( a_1, a_2) \in \mathcal{A}_1 = \mathbb{F}_{11} \times \mathbb{F}_{11}$, we proceed the following three steps:
\begin{description}
	\item[{\rm (1)}] Compute $h:= (P Q)^{p-1}$ over $\mathbb{F}_{11} [a_3, \ldots, a_{10}]$, where $a_3, \ldots, a_{10}$ are indeterminates. 
	\item[{\rm (2)}] We set $s_2 := 6$, and $(i_1, \ldots , i_{s_2}) := ( 4, 5, 6, 8, 9, 10)$ (we regard the $6$ coefficients $a_4$, $a_5$, $a_6$, $a_8$, $a_9$, $a_{10}$ as indeterminates).
	For solving multivariate systems over $\mathbb{F}_{11} [ a_4, a_5, a_6, a_8, a_9, a_{10}]$ in the next step, we adopt the graded reverse lexicographic (grevlex) order with
\[
a_{10} \prec a_9 \prec a_4 \prec a_8 \prec a_6 \prec a_5,
\]
	whereas for $\mathbb{F}_{11} [ x, y, z, w]$, the grevlex order with $w \prec z \prec y \prec x$ is adopted.
	Put $\mathcal{A}_2 := \mathbb{F}_{11} \times \mathbb{F}_{11}$.
	\item[{\rm (3)}] As in the case (i) of the proof of Proposition \ref{prop:Degenerate_q5}, we compute cubic forms $P$ such that $V(P,Q)$ are superspecial.
More precisely, we proceed the following three steps for each $( a_3, a_7 ) \in \mathcal{A}_2= \mathbb{F}_{11} \times \mathbb{F}_{11}$: 
	\begin{description}
		\item[{\rm (a)}] Let $\mathcal{S}$ be the set of the coefficients of the monomials of $\mathcal{M}$ in $h$, where the set $\mathcal{M}$ consists of the $16$ monomials in Corollary \ref{cor:HW}.
		Note that $\mathcal{S} \subset \mathbb{F}_{11} [a_4, a_5, a_6, a_8, a_9, a_{10} ]$.
		\item[{\rm (b)}] Solve the multivariate system $f ( a_4, a_5, a_6, a_8, a_9, a_{10} ) = 0$ for all $f \in \mathcal{S}$ over $\mathbb{F}_{11}$ with known algorithms via the Gr\"{o}bner basis computation.
		\item[{\rm (c)}] For each solution $(a_4, a_5, a_6, a_8, a_9, a_{10} )$ of the above system, substitute it into unknown coefficients in $P$, and decide whether $C = V ( P, Q)$ is non-singular or not. 
	\end{description}
\end{description}

\item[{\bf (ii) Case of $b_1 = 0$ and $a_4 \neq 0$}.]
For each $b_2 \in \{ 0, 1 \}$, execute Modified Version of Main Algorithm in \cite{KH16}, given in Section \ref{subsec:algorithm}.
We here give an outline of our computation together with our choices of $s_1$, $s_2$, $\{ k_1, \ldots , k_{s_1} \}$, $\{ i_1, \ldots , i_{s_2} \}$, $\mathcal{A}_1$, $\mathcal{A}_2$ and a term ordering in the algorithm.
\begin{description}
\item[{\rm (0)}]
We set $s_1:= 9$, and $( k_1, \ldots , k_{s_1} ):= ( 2, \ldots , 10)$ (we regard the $9$ coefficients $a_2$, $a_3$, $a_4$, $a_5$, $a_6$, $a_7$, $a_8$, $a_9$, $a_{10}$ as indeterminates).
Let $\mathcal{A}_1:= \mathbb{F}_{11}$.
\end{description}
For each $a_1 \in \mathcal{A}_1 = \mathbb{F}_{11}$, we proceed the following three steps:
\begin{description}
	\item[{\rm (1)}] Compute $h:= (P Q)^{p-1}$ over $\mathbb{F}_{11} [a_2, \ldots, a_{10}]$, where $a_2, \ldots, a_{10}$ are indeterminates. 
	\item[{\rm (2)}] Put $s_2 := 5$, and $(i_1, \ldots , i_{s_2}) := ( 5, 6, 8, 9, 10)$ (we regard the $5$ coefficients $a_5$, $a_6$, $a_8$, $a_9$, $a_{10}$ as indeterminates).
	For solving multivariate systems over $\mathbb{F}_{11} [ a_5, a_6, a_8, a_9, a_{10}]$ in the next step, we adopt the grevlex order with
\[
a_{10} \prec a_9 \prec a_8 \prec a_6 \prec a_5,
\]
	whereas for $\mathbb{F}_{11} [ x, y, z, w]$, the grevlex order with $w \prec z \prec y \prec x$ is adopted.
	Put $\mathcal{A}_2 := \mathbb{F}_{11} \times \mathbb{F}_{11} \times \mathbb{F}_{11}^{\times} \times \mathbb{F}_{11}$.
	\item[{\rm (3)}] We conduct a procedure similar to Step 3 in Case of $b_1 \neq 0$.
Specifically for each $( a_2, a_3, a_4, a_7 ) \in \mathcal{A}_2$, enumerate $(a_5, a_6, a_8, a_9, a_{10} )$ such that $C=V(P,Q)$ is superspecial.
\end{description}

\item[{\bf (iii) Case of $b_1 = 0$ and $a_4 = 0$}.]
For each $b_2 \in \{ 0, 1 \}$, execute Modified Version of Main Algorithm in \cite{KH16}, given in Section \ref{subsec:algorithm}.
We here give an outline of our computation together with our choices of $s_1$, $s_2$, $\{ k_1, \ldots , k_{s_1} \}$, $\{ i_1, \ldots , i_{s_2} \}$, $\mathcal{A}_1$, $\mathcal{A}_2$ and a term ordering in the algorithm.
\begin{description}
\item[{\rm (0)}]
We set $s_1:= 8$, and $( k_1, \ldots , k_{s_1} ):= ( 2, 3, 5, 6, \ldots , 10)$ (we regard the $8$ coefficients $a_2$, $a_3$, $a_5$, $a_6$, $a_7$, $a_8$, $a_9$, $a_{10}$ as indeterminates).
Let $\mathcal{A}_1:=\mathbb{F}_{11}$.
\end{description}
For each $a_1 \in \mathcal{A}_1 = \mathbb{F}_{11} $, we proceed the following three steps:
\begin{description}
	\item[{\rm (1)}] Compute $h:= (P Q)^{p-1}$ over $\mathbb{F}_{11} [a_2, a_3, a_5, a_6, \ldots, a_{10}]$, where $a_2, a_3, a_5, a_6, \ldots, a_{10}$ are indeterminates. 
	\item[{\rm (2)}] Put $s_2 := 4$, and $(i_1, \ldots , i_{s_2}) := ( 6, 8, 9, 10)$ (we regard the $4$ coefficients $a_6$, $a_8$, $a_9$, $a_{10}$ as indeterminates).
	For solving multivariate systems over $\mathbb{F}_{11} [ a_6, a_8, a_9, a_{10}]$ in the next step, we adopt the grevlex order with
\[
a_{10} \prec a_9 \prec a_8 \prec a_6,
\]
	whereas for $\mathbb{F}_{11} [ x, y, z, w]$, the grevlex order with $w \prec z \prec y \prec x$ is adopted.
	Put $\mathcal{A}_2 := \mathbb{F}_{11} \times \mathbb{F}_{11} \times \mathbb{F}_{11} \times \mathbb{F}_{11}$.
	\item[{\rm (3)}] We conduct a procedure similar to Step 3 in Case of $b_1 \neq 0$.
Specifically for each $( a_2, a_3, a_5, a_7 ) \in \mathcal{A}_2$, enumerate $(a_6, a_8, a_9, a_{10} )$ such that $C=V(P,Q)$ is superspecial.
\end{description}
\end{description}
Let $\mathcal{P}$ be the list of cubics $P$ such that $V(P,Q)$ was determined to be superspecial in the above procedures (i), (ii) and (iii).
For the inputs $\mathcal{P}$ and $q=11$, we execute Collecting Isomorphism Classes Algorithm for (N1), given in Section \ref{subsec:isom}.
By the outputs of our computation, a cubic form $P$ of the form \eqref{eq:N1q11} such that $V(P,Q)$ is superspecial is one of $P_i^{({\rm N1})}$ for $1 \leq i \leq 7$, up to isomorphism over $\mathbb{F}_{11}$.
\end{proof}

\subsubsection{Case of (N2) for $q = 11$}\label{subsubsec:comp_N2_q11}

\begin{prop}\label{prop:N2q11}
Consider the quadratic form $Q = 2 x w + y^2 - \epsilon z^2 \in \mathbb{F}_{11}$ with $\mathbb{F}_{11}^{\times} \smallsetminus ( \mathbb{F}_{11}^{\times} )^2$ and cubic forms of the form
\begin{equation}
\begin{split}
P = & ( a_1 y + a_2 z ) x^2 + a_3 (y^2 - \epsilon z^2) x + b_1 y ( y^2 - \epsilon z^2 ) + a_4 y ( y^2 + 3 \epsilon z^2 ) + a_5 z (3 y^2 + \epsilon z^2 ) \\
& + ( a_6 y^2 + a_7 y z + b_2 z^2 ) w + ( a_8 y + a_9 z ) w^2 + a_{10} w^3,
\end{split}\label{eq:N2q11}
\end{equation}
where $(a_1, a_2) \in (\mathbb{F}_{11} \times \mathbb{F}_{11}) \smallsetminus \{ (0, 0) \}$ and $b_1, b_2 \in \{ 0, 1 \}$.
Then a cubic form $P$ of the form \eqref{eq:N2q11} such that $V(P,Q)$ is superspecial is one of
\begin{eqnarray}
P_1^{({\rm N2})}& = & x^2 y + x^2 z + x y^2 + 9 x z^2 + 6 y^3 + y^2 z + 5 y^2 w + 3 y z^2 + 9 y w^2 + 8 z^3 + z^2 w + 9 z w^2 + 6 w^3, \nonumber \\
P_2^{({\rm N2})}&= & x^2 z + 5 y^3 + 4 z w^2, \nonumber \\
P_3^{({\rm N2})}&= & x^2 y + x^2 z + 9 y^3 + 8 y^2 z + 2 y z^2 + 4 y w^2 + 9 z^3 + 4 z w^2, \nonumber \\
P_4^{({\rm N2})}&= & 8 x^2 y + 2 x^2 z + y^3 + 8 y^2 z + 6 y^2 w + 9 y z^2 + 2 y z w + 5 y w^2 + 9 z^3 + z^2 w + 4 z w^2 + w^3, \nonumber \\
P_5^{({\rm N2})}&= & 6 x^2 y + 4 x^2 z + 6 x y^2 + 10 x z^2 + 10 y^3 + 4 y^2 z + 3 y^2 w + 8 y z^2 + 6 y z w + 9 y w^2 + 10 z^3 \nonumber \\
& & + z^2 w + z w^2 + 9 w^3, \nonumber
\end{eqnarray}
up to isomorphism over $\mathbb{F}_{11}$.
\end{prop}

\begin{proof}
Put $t = 10$, $u = 2$ and
\begin{eqnarray}
\{ p_1, \ldots , p_t \} & = & \{ y x^2, z x^2, ( y^2 - \epsilon z^2) x, y ( y^2 + 3 \epsilon z^2 ), z (3 y^2 + \epsilon z^2 ), y^2 w , y z w, y w^2, z w^2, w^3 \}, \nonumber \\
\{ q_1, \ldots , q_u \} & = & \{ y ( y^2 - \epsilon z^2 ), z^2 w \}. \nonumber 
\end{eqnarray}
For each $(b_1, b_2 ) \in \{0, 1 \}^{\oplus 2}$, execute Modified Version of Main Algorithm in \cite{KH16}, given in Section \ref{subsec:algorithm}.
We here give an outline of our computation together with our choices of $s_1$, $s_2$, $\{ k_1, \ldots , k_{s_1} \}$, $\{ i_1, \ldots , i_{s_2} \}$, $\mathcal{A}_1$, $\mathcal{A}_2$ and a term ordering in the algorithm.
\begin{description}
\item[{\rm (0)}]
We set $s_1:= 9$, and $( k_1, \ldots , k_{s_1} ):= ( 1, 2, 4, \ldots , 10)$ (we regard the $9$ coefficients $a_1$, $a_2$, $a_4$, $a_5$, $a_6$, $a_7$, $a_8$, $a_9$, $a_{10}$ as indeterminates).
Let $\mathcal{A}_1:=\mathbb{F}_{11}$.
\end{description}
For each $a_3 \in \mathcal{A}_1 = \mathbb{F}_{11}$, we proceed the following three steps:
\begin{description}
	\item[{\rm (1)}] Compute $h:= (P Q)^{p-1}$ over $\mathbb{F}_{11} [a_1, a_2, a_4, \ldots, a_{10}]$, where $a_1, a_2, a_4, \ldots, a_{10}$ are indeterminates. 
	\item[{\rm (2)}] Put $s_2 := 5$, and $(i_1, \ldots , i_{s_2}) := ( 6, 7, 8, 9, 10)$ (we regard the $5$ coefficients $a_6$, $a_7$, $a_8$, $a_9$, $a_{10}$ as indeterminates).
	 For solving multivariate systems over $\mathbb{F}_{11} [ a_6, a_7, a_8, a_9, a_{10}]$ in the next step, we adopt the grevlex order with
\[
a_{10} \prec a_9 \prec a_8 \prec a_7 \prec a_6,
\]
	whereas for $\mathbb{F}_{11} [ x, y, z, w]$, the grevlex order with $w \prec z \prec y \prec x$ is adopted.
	Put $\mathcal{A}_2 := (\mathbb{F}_{11} \times \mathbb{F}_{11} \smallsetminus \{ (0, 0 ) \} ) \times \mathbb{F}_{11} \times \mathbb{F}_{11}$.
	\item[{\rm (3)}] We conduct a procedure similar to Case of $b_1 \neq 0$ in Proposition \ref{prop:N1q11}.
Specifically for each $( a_1, a_2, a_4, a_5 ) \in \mathcal{A}_2$, enumerate $( a_6, a_7, a_8, a_9, a_{10} )$ such that $C=V(P,Q)$ is superspecial.
\end{description}
Let $\mathcal{P}$ be the list of cubics $P$ such that $V(P,Q)$ was determined to be superspecial in the above procedures.
For the inputs $\mathcal{P}$ and $q=11$, we execute a variant of Isomorphism Classes Collecting Algorithm given in Section \ref{subsec:isom} (for constructing the variant, see Remark \ref{rem:isom} (2)).
By the outputs of our computation, a cubic form $P$ of the form \eqref{eq:N2q11} such that $V(P,Q)$ is superspecial is one of $P_i^{({\rm N2})}$ for $1 \leq i \leq 5$, up to isomorphism over $\mathbb{F}_{11}$.
\end{proof}


\subsubsection{Degenerate case for $q=11$}\label{subsubsec:comp_dege_q11}

\begin{prop}\label{prop:Degenerate_q11}
Consider the quadratic form $Q = 2 y w + z^2 \in \mathbb{F}_{11}[x,y,z,w]$ and cubic forms $P \in \mathbb{F}_{11}[x,y,z,w]$ of the form
\begin{equation}
\begin{split}
P =& a_0 x^3 + ( a_1 y^2 + a_2 z^2 + a_3 w^2 + a_4 y z  + a_5 z w ) x  \\
& + a_6 y^3 + a_7 z^3 + a_8 w^3 + a_9 y z^2 + b_1 z^2 w + b_2 z w^2,
\end{split}\label{eq:Degeq11}
\end{equation}
where $a_0, a_6 \in \mathbb{F}_{11}^{\times}$, $b_1, b_2 \in \{ 0, 1 \}$.
Then a cubic form $P$ of the form \eqref{eq:Degeq11} such that $V(P,Q)$ is superspecial is one of
\begin{eqnarray}
P_1^{({\rm Dege})}&= & w^3+x^3+y^3, \nonumber \\
P_2^{({\rm Dege})}&= & 2 w^3+x^3+y^3, \nonumber \\
P_3^{({\rm Dege})}&= & 5 w^3+x^3+y^3+z^3, \nonumber \\
P_4^{({\rm Dege})}&= & w^2 x+x^3+y^3, \nonumber \\
P_5^{({\rm Dege})}&= & 2 w^2 x+x^3+y^3, \nonumber \\
P_6^{({\rm Dege})}&= & w^3+w x z+x^3+y^3+7 z^3, \nonumber \\
P_7^{({\rm Dege})}&= & 4 w^3+w^2 x+x^3+x y z+y^3+5 z^3, \nonumber \\
P_8^{({\rm Dege})}&= & 8 w^3+6 w^2 x+x^3+x y z+y^3+8 z^3, \nonumber \\
P_9^{({\rm Dege})}&= & 4 w^3+w^2 z+x^3+5 y^3+2 y z^2+z^3, \nonumber \\
P_{10}^{({\rm Dege})}&= & 2 w^3+w z^2+x^3+y^3+8 y z^2, \nonumber \\
P_{11}^{({\rm Dege})} &= & 3 w^3+w z^2+x^3+2 y^3+2 y z^2+4 z^3, \nonumber \\
P_{12}^{({\rm Dege})} &= & 10 w^3+w z^2+x^3+2 y^3+4 y z^2, \nonumber \\
P_{13}^{({\rm Dege})}&= & 7 w^3+w^2 z+w z^2+x^3+2 y^3+4 y z^2+z^3, \nonumber \\
P_{14}^{({\rm Dege})}&= & 7 w^3+2 w^2 x+w^2 z+8 w x z+w z^2+x^3+x y^2+7 x y z+8 x z^2+2 y^3+4 y z^2+z^3, \nonumber \\
P_{15}^{({\rm Dege})} &= & 10 w^3+w^2 z+w z^2+x^3+5 y^3+3 y z^2+5 z^3, \nonumber \\
P_{16}^{({\rm Dege})}&= & 6 w^3+w^2 z+w z^2+x^3+6 y^3+2 y z^2+6 z^3, \nonumber \\
P_{17}^{({\rm Dege})}& =& w^2 z+w z^2+x^3+10 y^3+6 y z^2+7 z^3, \nonumber 
\end{eqnarray}
up to isomorphism over $\mathbb{F}_{11}$.
\end{prop}

\begin{proof}
Put $t = 10$, $u = 2$ and
\begin{eqnarray}
\{ p_1, \ldots , p_t \} & = & \{ x^3, x y^2, x z^2, x w^2, x y z, x z w, y^3, z^3, w^3, y z^2 \}, \nonumber \\
\{ q_1, \ldots , q_u \} & = & \{ z^2 w, z w^2 \}. \nonumber 
\end{eqnarray}
For each $(b_1, b_2 ) \in \{ 0,1 \}^{\oplus 2}$, execute Modified Version of Main Algorithm in \cite{KH16}, given in Section \ref{subsec:algorithm}.
We here give an outline of our computation together with our choices of $s_1$, $s_2$, $\{ k_1, \ldots , k_{s_1} \}$, $\{ i_1, \ldots , i_{s_2} \}$, $\mathcal{A}_1$, $\mathcal{A}_2$ and a term ordering in the algorithm.
\begin{description}
\item[{\rm (0)}]
We set $s_1:= 9$, and $( k_1, \ldots , k_{s_1} ):= ( 1, \ldots , 9)$ (we regard the $9$ coefficients $a_1$, $a_2$, $a_3$, $a_4$, $a_5$, $a_6$, $a_7$, $a_8$, $a_9$ as indeterminates).
Let $\mathcal{A}_1:=\mathbb{F}_{11}^{\times}$.
\end{description}
For each $a_0 \in \mathcal{A}_1:= \mathbb{F}_{11}^{\times}$, we proceed the following three steps:
\begin{description}
	\item[{\rm (1)}] Compute $h:= (P Q)^{p-1}$ over $\mathbb{F}_{11} [a_1, \ldots, a_9]$, where $a_1, \ldots, a_9$ are indeterminates. 
	\item[{\rm (2)}] Put $s_2 := 6$, and $(i_1, \ldots , i_{s_2}) := ( 2, 3, 5, 7, 8, 9)$ (we regard the $6$ coefficients $a_2$, $a_3$, $a_5$, $a_7$, $a_8$, $a_9$ as indeterminates).
	For solving multivariate systems over $\mathbb{F}_{11}[a_2, a_3, a_5, a_7, a_8, a_9]$, we adopt the grevlex order with
\[
a_8 \prec a_7 \prec a_9 \prec a_3 \prec a_5 \prec a_2,
\]
	whereas for $\mathbb{F}_{11} [ x, y, z, w]$, we adopt the grevlex order with $w \prec z \prec y \prec x$.
	Put $\mathcal{A}_2 := \mathbb{F}_{11} \times \mathbb{F}_{11} \times \mathbb{F}_{11}^{\times}$.
	\item[{\rm (3)}] We conduct a procedure similar to Case of $b_1 \neq 0$ in Proposition \ref{prop:N1q11}.
Specifically for each $( a_1, a_4, a_6 ) \in \mathcal{A}_2$, enumerate $( a_2, a_3, a_5, a_7, a_8, a_9 )$ such that $C=V(P,Q)$ is superspecial.
\end{description}
Let $\mathcal{P}$ be the list of cubics $P$ such that $V(P,Q)$ was determined to be superspecial in the above procedures.
For the inputs $\mathcal{P}$ and $q=11$, we execute a variant of Collecting Isomorphism Classes Algorithm given in Section \ref{subsec:isom} (for constructing the variant, see Remark \ref{rem:isom} (2)).
By the outputs of our computation, a cubic form $P$ of the form \eqref{eq:Degeq11} such that $V(P,Q)$ is superspecial is one of $P_i^{({\rm Dege})}$ for $1 \leq i \leq 17$, up to isomorphism over $\mathbb{F}_{11}$.
\end{proof}


\subsubsection{Computational parts for our proof of Corollary \ref{MainCorollary}}

\begin{prop}\label{prop:for_cor}
Consider the quadratic forms $Q^{(\mathrm{N1})} = 2 x w + 2 y z$, $Q^{(\mathrm{N2})} = 2 x w + y^2 - \epsilon z^2$ with $\mathbb{F}_{11}^{\times} \smallsetminus ( \mathbb{F}_{11}^{\times} )^2$, $Q^{(\mathrm{Dege})} = 2 y w + z^2$ and the cubic forms
\begin{eqnarray}
P_1^{(\mathrm{alc})}&:=& x^2 y + x^2 z + 2 y^2 z + 5 y^2 w + 9 y z^2 + y z w + 4 z^3 + 3 z^2 w + 10 z w^2 + w^3, \nonumber \\
P_2^{(\mathrm{alc})} &:=& x^2 y + x^2 z + y^3 + y^2 z + 7 y z^2 + 4 y w^2 + 2 z^3 + 9 z w^2, \nonumber \\
P_3^{(\mathrm{alc})} &:=& x^2 y + x^2 z + y^3 + 8 y^2 z + 3 y z^2 + 10 y w^2 + 10 z^3 + 10 z w^2, \nonumber \\
P_4^{(\mathrm{alc})} &:=& x^3 + y^3 + w^3, \nonumber \\
P_5^{(\mathrm{alc})} &:=& x^3 + y^3 + z^3 + 5 w^3, \nonumber \\
P_6^{(\mathrm{alc})} &:=& x^3 + x w^2 + y^3, \nonumber \\
P_7^{(\mathrm{alc})} &:=& x^3 + x z w + y^3 + 7 z^3 + w^3, \nonumber \\
P_8^{(\mathrm{alc})} &:=& x^3 + x y z + x w^2 + y^3 + 5 z^3 + 4 w^3, \nonumber \\
P_9^{(\mathrm{alc})} &:=& x^3 + x y z + 6 x w^2 + y^3 + 8 z^3 + 8 w^3 \nonumber 
\end{eqnarray}
in $\mathbb{F}_{11}[x,y,z,w]$.
Let $P_i^{(\mathrm{N1})}$, $P_j^{(\mathrm{N2})}$ and $P_k^{(\mathrm{Dege})}$ be as in $\mathrm{Propositions}$ $\ref{prop:N1q11}$ -- $\ref{prop:Degenerate_q11}$.
\begin{enumerate}
\item[{\rm (1)}] Each of $V (P_i^{(\mathrm{N1})}, Q^{(\mathrm{N1})})$ and $V (P_j^{(\mathrm{N2})}, Q^{(\mathrm{N2})} )$ with $1 \leq i \leq 8$ and $1 \leq j \leq 5$ is isomorphic to $V ( P_k^{(\mathrm{alc})}, Q^{(\mathrm{N1})})$ over $\overline{\mathbb{F}_{11}}$ for some $1 \leq k \leq 3$, and vice versa.
Moreover, $V ( P_i^{(\mathrm{alc})}, Q^{(\mathrm{N1})})$ is not isomorphic to $V ( P_j^{(\mathrm{alc})}, Q^{(\mathrm{N1})})$ over $\overline{\mathbb{F}_{11}}$ for each $1 \leq i < j \leq 3$.
\item[{\rm (2)}] Each $V (P_i^{(\mathrm{Dege})}, Q^{(\mathrm{Dege})})$ with $1 \leq i \leq 17$ is isomorphic to $V ( P_j^{(\mathrm{alc})}, Q^{(\mathrm{Dege})})$ over $\overline{\mathbb{F}_{11}}$ for some $4 \leq j \leq 9$, and vice versa.
Moreover, $V ( P_i^{(\mathrm{alc})}, Q^{(\mathrm{Dege})})$ is not isomorphic to $V ( P_j^{(\mathrm{alc})}, Q^{(\mathrm{Dege})})$ over $\overline{\mathbb{F}_{11}}$ for each $4 \leq i < j \leq 9$.
\end{enumerate}
\end{prop}
\begin{proof}
We prove (1).
Take $\epsilon = 2 \in \mathbb{F}_{11}$.
We first transform $V ( P_i^{(\mathrm{N2})}, Q^{(\mathrm{N2})})$ into $V ( P, Q^{(\mathrm{N1})})$ for some cubic form $P$ by the actions of elements in $\mathrm{GL}_4 ( \overline{\mathbb{F}_{11}})$.
Put
\[
M_Q := 
\begin{pmatrix}
1 & 0 & 0 & 0\\
0 & 1/2 & 1 & 0\\
0 & 1/2 \sqrt{\epsilon} & -1/\sqrt{\epsilon} & 0\\
0 & 0 & 0 & 1
\end{pmatrix}
\]
and $P_{8+i}^{(\mathrm{N1})}:=M_Q \cdot P_i^{(\mathrm{N2})} \ (\mbox{mod } Q^{({\rm N1})})$, where ``$(\mbox{mod } Q^{({\rm N1})})$'' means here replacing $x w$ in $M_Q \cdot P_i^{(\mathrm{N2})}$ by $- 2^{-1} (Q^{({\rm N1})} - 2 x w)$ via $x w \equiv - 2^{-1} (Q^{({\rm N1})} - 2 x w) \mbox{ mod }Q^{({\rm N1})}$.
Note that each $V ( P_i^{(\mathrm{N2})}, Q^{(\mathrm{N2})} )$ is isomorphic to $V ( P_{8+i}^{(\mathrm{N1})}, Q^{(\mathrm{N1})})$ over $\overline{\mathbb{F}_{11}}$.

For $\mathcal{P}:= ( P_1^{(\mathrm{N1})}, \ldots , P_8^{(\mathrm{N1})}, P_{9}^{(\mathrm{N1})}, \ldots , P_{13}^{(\mathrm{N1})} )$, we conduct a variant of Collecting Isomorphism Classes Algorithm given in the third paragraph of Section \ref{subsec:algorithm}.
By the outputs of our computations, we have that each $V (P_i^{(\mathrm{N1})}, Q^{(\mathrm{N1})})$ is isomorphic to $V ( P_j^{(\mathrm{alc})}, Q^{(\mathrm{N1})})$ over $\overline{\mathbb{F}_{11}}$ for some $1 \leq j \leq 3$, and vice versa.
The outputs also show that $V ( P_i^{(\mathrm{alc})}, Q^{(\mathrm{N1})})$ is not isomorphic to $V ( P_j^{(\mathrm{alc})}, Q^{(\mathrm{N1})})$ over $\overline{\mathbb{F}_{11}}$ for each $1 \leq i < j \leq 3$.

(2) is proved by a computation similar to (1).
\end{proof}

\subsection{Our implementation to prove the main theorems}\label{subsec:imple}
The algorithms given in Sections \ref{subsec:algorithm} and \ref{subsec:isom} are
\begin{itemize}
\item Modified Version of Main Algorithm in \cite{KH16} (pseudocode: Algorithm \ref{alg:enume}),
\item Isomorphism Testing Algorithm (pseudocode: Algorithm \ref{alg:isomorphic_tilde}), and
\item Collecting Isomorphism Classes Algorithm (pseudocode: Algorithm \ref{alg:isomorphic_family}).
\end{itemize}
Recall that Isomorphism Testing Algorithm is a sub-procedure in Collecting Isomorphism Classes Algorithm.
The source codes and the log files are available at the web page of the first author \cite{HPkudo}.
In this subsection, we show timing and sample codes for the case of $q =11$.

\subsubsection{Timing}\label{subsubsec:timing}

We measured time used in both of Modified Version of Main Algorithm in \cite{KH16} and Collecting Isomorphism Classes Algorithm for each case.
We show the timing only for Modified Version of Main Algorithm in \cite{KH16} since most of time was used in this step.
Table \ref{timing} shows the timing of Modified Version of Main Algorithm in \cite{KH16} in our computation to show Propositions \ref{prop:N1q11} -- \ref{prop:Degenerate_q11}.
\begin{table}[h]
  \begin{center}
    \caption{Timing data of Modified Version of Main Algorithm in \cite{KH16} (pseudocode: Algorithm \ref{alg:enume}) in our computation to prove Propositions \ref{prop:N1q11} -- \ref{prop:Degenerate_q11}}\label{timing}
\vspace{0.2cm}
    \scalebox{0.8}{
    \begin{tabular}{c|l|c|c|l|c|l|l|r|r} \hline
    $q$ & Case & $s_1$ & Iterations 1 & ~~~~~$t_{mlt}$ & $s_2$ & ~~~~~$t_{GBslv}$ & ~~~~~$t_{total}$ & Iterations 2 & Total time~~~ \\ \hline
$11$ & {\bf N1} (i) & $8$ & 484 & 0.59267s & $6$ & 0.88283s & 0.88361s & 58564 & 52067.856s \\
 & & & & & & & & & (about 14.5 hours) \\
 & {\bf N1} (ii) & $9$ & 22 & 0.39381s & $5$ & 0.036043s  & 0.039275s & 292820 & 11901.708s\\ 
 & {\bf N1} (iii) & $8$ & 3456  & 0.19923s & $4$ & 0.011823s & 0.014758s & 322102& 5130.876s \\
 & {\bf N2} & $9$ & 22 & 9.1519s & $5$ & 0.16173s & 0.17191s & 638880 & 113056.172s  \\
 & & & & & & & & & (about 31.4 hours) \\
 & {\bf Dege} & $9$ & 40 & 0.39938s & $6$ & 0.13430s & 0.13676s & 48400 & 6645.091s \\ \hline
	 \end{tabular}}
 \end{center}
\end{table}

\paragraph{Table Notation}
Let $q$ denote the cardinality of $K$.
``Iterations 1'' denotes the number of iterations on $b_i$ and $a_j$ which are not regarded as indeterminates at Step (1) of Modified Version of Main Algorithm in \cite{KH16}.
Let ``$s_1$'' denote the number of indeterminates in the multiplication $(PQ)^{p-1}$ in Step (1) for each case.
We denote by ``$t_{mlt}$'' the time used in Step (1) for computing $(P Q)^{p-1}$.
Note that we regard $s_1$ coefficients in $P$ as indeterminates and thus this computation is done over a multivariate polynomial ring with the indeterminates $x$, $y$, $z$, $w$ whose ground ring is a polynomial ring of $s_1$ indeterminates.
Let ``$s_2$'' denote the number of indeterminates in the computation to solve a multivariate system in Step (3b) for each case.
The notation``$t_{GBslv}$'' is the time used in Step (3b) for solving a multivariate system over $K$.
``Iterations 2'' denotes the number of iterations on $b_i$ and $a_j$ which are not regarded as indeterminates at Step (3b).
Let ``$t_{total}$'' denote the time used in Steps (3a)-(3c) for each iteration, whereas ``Total time'' denotes the total time used in Steps (1)-(3), namely the total time taken for each case.

\paragraph{Workstation.}
We conducted the computation to prove Propositions \ref{prop:Degenerate_q5} -- \ref{prop:Degenerate_q11} by a Windows 10 home OS, 64 bit computer with 3.40 GHz CPU (Intel Core i7) and 20 GB memory.
In our implementation, the following built-in functions in Magma are called:
\begin{enumerate}
	\item \texttt{GroebnerBasis}:
Let $K$ be a computable field.
For given polynomials $f_1, \ldots , f_s$ in $R:=K [ X_1, \ldots , X_n]$, this function outputs the (reduced) Gr\"{o}bner basis of the ideal $\langle f_1, \ldots , f_s \rangle_R$ with respect to a decidable term order $\succ$. 
	\item \texttt{Variety}: Let $K$ be a computable field and $R:=K [ X_1, \ldots , X_n]$ the polynomial ring in $n$ indeterminates over $K$.
For given polynomials $f_1, \ldots , f_s \in R$ such that $V_{\overline{K}}(f_1, \ldots , f_s) = \{ (a_1, \ldots , a_n) \in \overline{K}^n : f_i (a_1, \ldots , a_n) = 0 \mbox{ for all } 1 \leq i \leq s \}$ is finite, this function outputs $V_{K}(f_1, \ldots , f_s) = \{ (a_1, \ldots , a_n) \in K^n : f_i (a_1, \ldots , a_n) = 0 \mbox{ for all } 1 \leq i \leq s \}$.
Note that this function also works for higher-dimensional ideals if $K$ is finite.
\end{enumerate}
As we will show in a piece of our codes in the next subsection, we implemented the following function as a sub-routine.
\begin{enumerate}
\item[(3)] \texttt{RestrictedVariety}: Let $K$ be a computable field and $R=K[ X_1, \ldots, X_n ]$ the polynomial ring in $n$ indeterminates over $K$.
Let $f_1, \ldots , f_s \in R$ be polynomials such that $V_{\overline{K}}(f_1, \ldots , f_s) = \{ (a_1, \ldots , a_n) \in \overline{K}^n : f_j (a_1, \ldots , a_n) = 0 \mbox{ for all } 1 \leq j \leq s \}$ is finite.
Let $\underline{i} = ( i_1, \ldots , i_n )$ be a sequence with $0$ or $1$ entries, and $k_1 < \cdots < k_{n-h(\underline{i})}$ indexes on the $0$ entries $i_k = 0$ of $\underline{i}$, where $h(\underline{i})$ denotes the Hamming weight of $\underline{i}$.
Let $c = (c_1, \ldots , c_{n-h(\underline{i})}) \in K^{n-h(\underline{i})}$ be a tuple.
Given $f_1, \ldots , f_s$, $\underline{i}$ and $c$, this function computes
\[
\{ (a_1, \ldots , a_n) \in K^n : f_j (a_1, \ldots , a_n) = 0 \mbox{ for all } 1 \leq j \leq s, \mbox{ and } a_{k_{\ell}} = c_{\ell} \mbox{ for all } 1 \leq \ell \leq n - h (\underline{i}) \}
\]
by substituting $c_{\ell}$ into $X_{k_{\ell}}$ together with the built-in function $\texttt{Variety}$.
\end{enumerate}


\subsubsection{Codes for our computation}\label{subsubsec:implementation}

\paragraph{Loading our implementation program.}

Our codes are executed on Magma as follows.
Assume that the code file \texttt{ssp\_code\_q11N1\_b1nonzero\_v2.txt} is in the directory \texttt{C:/Users}. 

\begin{framed}
\vspace{-0.5cm}
\begin{small}
\begin{verbatim}
Magma V2.22-2     Wed Jan 25 2017 12:53:45 on DESKTOP-GN95KFP [Seed = 670522909]
Type ? for help.  Type <Ctrl>-D to quit.
> load"C:/Users/ssp_code_q11N1_b1nonzero_v2.txt";
\end{verbatim}
\end{small}
\vspace{-0.5cm}
\end{framed}

\paragraph{Sample code.}

In the following, a piece of our implementation codes is given.

\begin{framed}
\vspace{-0.5cm}
\begin{small}
\begin{verbatim}
Magma V2.22-2     Wed Jan 25 2017 12:53:45 on DESKTOP-GN95KFP [Seed = 670522909]
Type ? for help.  Type <Ctrl>-D to quit.
> /*******************************************/
> RestrictedVariety:= function(P,B,ind,tup);
function>       rank_PP:=Rank(P)-#tup;
function>       PP<[s]>:=PolynomialRing(CoefficientRing(P),rank_PP,"grevlex");
function>       P_new:=ChangeRing(P,PP);
function>       B0:=[P_new!(B[i]) : i in [1..#B]];
function>       B1:=[];
function>       v:=[];
function>       ii:=0;
function>       for i in [1..#ind] do
function|for>           if ind[i] eq 1 then
function|for|if>                        v[i]:=s[i-ii];
function|for|if>                else
function|for|if>                        ii:=ii+1;
function|for|if>                        v[i]:=PP!tup[ii];
function|for|if>                end if;
function|for>   end for;
function>       for i in [1..#B0] do
function|for>           B1[i]:=Evaluate(B0[i],v);
function|for>   end for;
function>       B_new:=B1;
function>       II:=ideal<PP|B_new>;
function>       VV:=Variety(II);
function>       V_out:=[];
function>       for i in [1..#VV] do
function|for>           vi:=<>;
function|for>           k:=0;
function|for>           for j in [1..#ind] do
function|for|for>                       if ind[j] eq 1 then
function|for|for|if>                            k:=k+1;
function|for|for|if>                            vi:=Append(vi,VV[i][k]);
function|for|for|if>                    else
function|for|for|if>                            vi:=Append(vi,tup[j-k]);
function|for|for|if>                    end if;
function|for|for>               end for;
function|for>           V_out[i]:=vi;
function|for>   end for;
function>       return V_out;
function> end function;
> //-------------------------------------------
> p:=11;
> q:=p;
> K:=GF(q);
> count_HW0:=0;
> b1:=K!1; b2:=K!0;
> a1:=K!1; a2:=K!0;
> s1:=8;
> s2:=6;
> R<[t]>:=PolynomialRing(K,s1,"grevlex");
> S<x,y,z,w>:=PolynomialRing(R,4,"grevlex");
> exponents_set:=[
> [ 2*p-2, p-1, p-1, p-1],
> [ 2*p-1, p-2, p-1, p-1],
> [ 2*p-1, p-1, p-2, p-1],
> [ 2*p-1, p-1, p-1, p-2],
> [ p-1, 2*p-2, p-1, p-1],
> [ p-2, 2*p-1, p-1, p-1],
> [ p-1, 2*p-1, p-2, p-1],
> [ p-1, 2*p-1, p-1, p-2],
> [ p-1, p-1, 2*p-2, p-1],
> [ p-2, p-1, 2*p-1, p-1],
> [ p-1, p-2, 2*p-1, p-1],
> [ p-1, p-1, 2*p-1, p-2],
> [ p-1, p-1, p-1, 2*p-2],
> [ p-2, p-1, p-1, 2*p-1],
> [ p-1, p-2, p-1, 2*p-1],
> [ p-1, p-1, p-2, 2*p-1]];
> not_vanished_monomials:=
> {@ x^(E[1])*y^(E[2])*z^(E[3])*w^(E[4]) : E in exponents_set @};
> Coeff_set:=MonomialsOfDegree(R,1);
> f:= x^2*y + b1*x^2*z + b2*x*z^2 + a1*y^3 + a2*y^2*z;
> g:= 2*x*w + 2*y*z;
> Mono_set_deg3_unknown:={@ y*z^2, y^2*w, y*z*w, z^2*w, y*w^2, z^3, z*w^2, w^3 @}; // 8
> // 8 = s1 - s2
> // a3, a5, a6, a7, a8, a4, a9, a10
> for i in [1..#Mono_set_deg3_unknown] do
for> f:= f + S!(Coeff_set[i])*(Mono_set_deg3_unknown[i]);
for> end for;
> f1:=f^(p-1);
> g1:=g^(p-1);
> h:=f1*g1;
> for a3 in K do
for> for a7 in K do
for|for> F:=[];
for|for> for i in [1..#(not_vanished_monomials)] do
for|for|for>     F[i]:=MonomialCoefficient(h,not_vanished_monomials[i]);
for|for|for> end for;
for|for> ind:=[0,1,1,0,1,1,1,1]; // a3, a5, a6, a7, a8, a4, a9, a10
for|for> tup:=[a3,a7];
for|for> V:=RestrictedVariety(R,F,ind,tup);
for|for> count_HW0:=count_HW0 + #V;
for|for> end for; // a7
for> end for; // a3
> count_HW0;
8
\end{verbatim}
\end{small}
\vspace{-0.5cm}
\end{framed}

In the above piece of codes, we compute roots of a multivariate system constructed from our criterion for superspecialty  for Case {\bf (N1)} (i) for $q=11$ and certain fixed coefficients.
Here the notation are the same as in Proposition \ref{prop:N1q11}.
For $( b_1, b_2) = (1, 0)$ and $(a_1, a_2) = ( 1, 0)$, we seek all the tuples $( a_3, a_4, a_5, a_6, a_7, a_8, a_9, a_{10}) \in (\mathbb{F}_{11})^{\oplus 8}$ of coefficients in $P$ such that the Hasse-Witt matrix of $C = V ( P, Q)$ is zero.
From the final output, one has that the number of roots $( a_3, \ldots , a_{10})$ is $8$.


\appendix


\section{Pseudocodes}\label{sec:code}

In this appendix, we collect the pseudocodes of the algorithms proposed in Section \ref{sec:main_results}.
For the notation in each code, see Section \ref{sec:main_results}.
The algorithms are
\begin{itemize}
\item Modified Version of Main Algorithm in \cite{KH16} (pseudocode: Algorithm \ref{alg:enume}),
\item Isomorphism Testing Algorithm (pseudocode: Algorithm \ref{alg:isomorphic_tilde}), and
\item Collecting Isomorphism Classes Algorithm (pseudocode: Algorithm \ref{alg:isomorphic_family}).
\end{itemize}
Recall that Isomorphism Testing Algorithm is a sub-procedure in Collecting Isomorphism Classes Algorithm.

\begin{algorithm}[htb] 
\caption{$\texttt{EnumerateSSpCurves} ( Q, P, p)$}
\label{alg:enume}
\begin{algorithmic}[1]
\REQUIRE{A quadratic form $Q$ in $S = \mathbb{F}_q [x, y, z, w]$, a cubic form $P$ of the form \eqref{input_P} in $\mathbb{F}_q [a_1, \ldots , a_t] [x,y,z,w]$, and the characteristic $p$ of $\mathbb{F}_q$}
\ENSURE{A list $\mathcal{P}$ of cubics $P$ such that the curves $C = V ( P, Q )$ are superspecial}
\STATE $\mathcal{P}$ $\leftarrow$ $\emptyset$
\STATE $\mathcal{M}$ $\leftarrow$ the set of the $16$ monomials given in Corollary \ref{cor:HW}
\STATE Choose $0 \leq s_1 \leq t$ and $\{ k_1, \ldots , k_{s_1} \} \subset \{ 1, \ldots , t \}$
\STATE $t_1$ $\leftarrow$ $t-s_1$; Write $\{ 1, \ldots , t \} \smallsetminus \{ k_1, \ldots , k_{s_1} \} = \{ k_1^{\prime}, \ldots , k_{t_1}^{\prime} \}$
\STATE Choose $\mathcal{A}_1 \subset \mathbb{F}_q^{\oplus t_1}$
\FOR{$( a_{k_1^{\prime}}, \ldots , a_{k_{t_1}^{\prime}} ) \in \mathcal{A}_1$}
	\STATE Substitute $( a_{k_1^{\prime}}, \ldots, a_{k_{t_1}^{\prime}} )$ to $P$ /* Keep $a_{k_1}, \ldots , a_{k_{s_1}}$ being indeterminates*/
	\STATE $h$ $\leftarrow$ $( P Q )^{p-1}$
	\STATE Choose $0 \leq s_2 \leq s_1$ and $\{ i_1, \ldots , i_{s_2} \} \subset \{ k_1, \ldots , k_{s_1} \}$
	\STATE $t_2$ $\leftarrow$ $s_1-s_2$; Write $\{ k_1, \ldots , k_{s_1} \} \smallsetminus \{ i_1, \ldots , i_{s_2} \} = \{ j_1, \ldots , j_{t_2} \}$
	\STATE Choose $\mathcal{A}_2 \subset \mathbb{F}_q^{\oplus t_2}$
	\FOR{$\left( a_{j_1}, \ldots , a_{j_{t_2}} \right) \in \mathcal{A}_2$}
		\STATE Substitute $\left( a_{j_1}, \ldots, a_{j_{t_2}} \right)$ to $P$ /* Keep $a_{i_1}, \ldots , a_{i_{s_2}}$ being indeterminates*/
		\STATE $\mathcal{S}$ $\leftarrow$ $\{ \mbox{the coefficient of } x^k y^{\ell} z^m w^n : x^k y^{\ell} z^m w^n \in \mathcal{M} \}t$
		\STATE $I$ $\leftarrow$ the ideal $\langle \mathcal{S} \rangle \subset \mathbb{F}_q [ a_{i_1}, \ldots, a_{i_{s_2}} ]$
		\STATE Choose a term ordering on $a_{i_1}, \ldots, a_{i_{s_2}}$
		\STATE Solve the system $f = 0$ for all $f \in \mathcal{S}$ over $K$ by some known algorithm with $\succ$
		\STATE $V$ $\leftarrow$ $V(I) = \{ ( a_{i_1}, \ldots , a_{i_{s_2}} ) \in \mathbb{F}_q^{\oplus s_2} : f \left( a_{i_1}, \ldots, a_{i_{s_2}} \right) = 0 \mbox{ for all } f \in \mathcal{S} \}$
		\IF{$V \neq \emptyset$}
			\FOR{$\left( a_{i_1}, \ldots, a_{i_{s_2}} \right) \in V$}
				\STATE Substitute $\left( a_{i_1}, \ldots, a_{i_{s_2}} \right)$ to $P$ /* Then $P \in \mathbb{F}_q[x,y,z,w]$ */
				\STATE Decide whether $V ( P, Q )$ is non-singular by the non-singularity testing in Section \ref{subsec:singtest}
				\IF{$V ( P, Q )$ is non-singular}
					\STATE $\mathcal{P}$ $\leftarrow$ $\mathcal{P} \cup \{ P \}$
				\ENDIF
			\ENDFOR
		\ENDIF
	\ENDFOR
\ENDFOR
\RETURN $\mathcal{P}$
\end{algorithmic}
\end{algorithm}

\paragraph{``Mod $Q$'' function}
Let $Q$ be an irreducible quadratic form in $K [x,y,z,w]$, $m_Q$ a monomial in $Q$, and $c_Q$ its coefficient.
Given a cubic form $P \in K [x,y,z,w]$, we give a function to compute ``$P \mbox{ mod } Q$'', where $\mbox{mod }Q$ means here replacing $m_Q$ in $P$ by $- c_Q^{-1} (Q - c_Q m_Q)$ via $m_Q \equiv - c_Q^{-1} (Q - c_Q m_Q) \mbox{ mod }Q$.
In other words, we compute a cubic form $P^{\prime}$ with no monomial divided by $m_Q$ such that $P \equiv P^{\prime} \mbox{ mod } Q$.
This function shall be used as a sub-routine in Algorithm \ref{alg:isomorphic_tilde}.
In Algorithm \ref{alg:modQ}, we write down a pseudocode of the function.

\begin{algorithm}[htb] %
\caption{$\texttt{ModQuad} ( P, Q, m_Q )$}
\label{alg:modQ}
\begin{algorithmic}[1]
\REQUIRE{A cubic form $P \in K [x,y,z,w]$, a quadratic form $Q$ and a monomial $m_Q$ in $Q$}
\ENSURE{A cubic form $P^{\prime}$ with no monomial divided by $m_Q$ such that $P \equiv P^{\prime} \mbox{mod } Q$}
\STATE $c_Q$ $\leftarrow$ the coefficient of $m_Q$ in $Q$
\STATE Write $P = P_0 + m_Q R$ for a cubic form $P_0$ with no term containing $m_Q$, and a linear form $R$
\STATE $P^{\prime}$ $\leftarrow$ $P_0 - c_Q^{-1} (Q - c_Q m_Q) R$
\RETURN $P^{\prime}$
\end{algorithmic}
\end{algorithm}

\begin{algorithm}[htb] %
\caption{$\texttt{IsIsomorphicN1} ( P_1, P_2, q )$}
\label{alg:isomorphic_tilde}
\begin{algorithmic}[1]
\REQUIRE{Two cubic forms $P_1$ and $P_2$ in $\mathbb{F}_q [x,y,z,w]$, and $q = p^s$ a power of a prime $p$}
\ENSURE{``\textbf{ISOMORPHIC}'' or ``\textbf{NOT ISOMORPHIC}''}
\STATE $IsomorphicFlag$ $\leftarrow$ $0$
\FOR{$M_{\rm A} \in \mathrm{A}$}
	\FOR{$M_{\rm W} \in \mathrm{W}$}
		\STATE $M_{\tilde{\rm T}}$ $\leftarrow$ $\mathrm{diag} (t_1, t_2, t_3 t_2^{-1}, t_3 t_1^{-1})$
		\STATE $g$ $\leftarrow$ $M_{\rm A} \cdot M_{\tilde{\rm T}} \cdot U_1 (t_4) \cdot U_2 (t_5) \cdot M_{\rm W} \cdot U_1 (t_6) \cdot U_2 (t_7)$
		\STATE Construct a system of algebraic equations with indeterminates $t_i$'s and $\lambda$
		\STATE $P_3$ $\leftarrow$ $\texttt{ModQuad} ( g \cdot P_1 - \lambda P_2, Q, x w)$
	 	\STATE $\mathcal{S}$ $\leftarrow$ $\{ t_i^{q-1} - 1 : 1 \leq i \leq 3 \} \cup \{ t_j^q - t_j : 4 \leq j \leq 7 \} \cup \{ \lambda^{q-1} - 1 \}$
	 	\STATE $\mathrm{Mon}(P_3)$ $\leftarrow$ the set of the monomials in $P_3$
		\FOR{$x^k y^{\ell} z^m w^n \in \mathrm{Mon}(P_3)$}
			\STATE $f  (t_1, t_2, t_3, t_4, t_5, t_6, t_7, \lambda)$ $\leftarrow$ the coefficient of $x^k y^{\ell} z^m w^n$ in $P_3$
			\STATE $\mathcal{S}$ $\leftarrow$ $\mathcal{S} \cup \{ f (t_1, t_2, t_3, t_4, t_5, t_6, t_7, \lambda) \}$
		\ENDFOR
		\STATE $G$ $\leftarrow$ the reduced Gr\"{o}bner basis for $\langle \mathcal{S} \rangle \subset \mathbb{F}_q [ t_1, t_2, t_3, t_4, t_5, t_6, t_7, \lambda ]$
		\IF{$\sharp G \neq 1$}
			\STATE $IsomorphicFlag$ $\leftarrow$ $1$, \textbf{break} $M_{\rm W}$ and $M_{\rm A}$
		\ENDIF
	\ENDFOR
\ENDFOR
\IF{$IsomorphicFlag = 0$}
	\RETURN ``\textbf{NOT ISOMORPHIC}''
\ELSE
	\RETURN ``\textbf{ISOMORPHIC}''
\ENDIF
\end{algorithmic}
\end{algorithm}

\begin{algorithm}[t] %
\caption{$\texttt{NotIsomorphicListN1} ( \mathcal{P}, q )$}
\label{alg:isomorphic_family}
\begin{algorithmic}[1]
\REQUIRE{A list $\mathcal{P} = ( P_1, \ldots , P_t )$ of cubics in $\mathbb{F}_q [x,y,z,w]$, and $q = p^s$ a power of a prime $p$}
\ENSURE{A family $\mathcal{P}^{\prime \prime}$ of cubics in $\mathbb{F}_q [x,y,z,w]$}
\STATE $\mathcal{P}^{\prime}$ $\leftarrow$ $\emptyset$, $FlagList1$ $\leftarrow$ $( 0 )_{i=1}^t$
\FOR{$i=1$ \TO $t$}
	\IF{$FlagList1[i] = 0$}
		\STATE $\mathcal{P}^{\prime}$ $\leftarrow$ $\mathcal{P}^{\prime} \cup \{ P_i \}$
		\FOR{$j=i+1$ \TO $t$}
			\STATE Use Algorithm \ref{alg:isomorphic_tilde} not regarding $t_3$ as an indeterminate, but taking $t_3=1$
			\IF{$\texttt{IsIsomorphicN1} ( P_i, P_j, q )$ returns \textbf{ISOMORPHIC}}
				\STATE $FlagList1[j]$ $\leftarrow$ $1$
			\ENDIF
		\ENDFOR
	\ENDIF
\ENDFOR
\STATE $\mathcal{P}^{\prime \prime}$ $\leftarrow$ $\emptyset$, $FlagList2$ $\leftarrow$ $( 0 )_{i=1}^{\sharp \mathcal{P}^{\prime}}$
\FOR{$i=1$ \TO $\sharp \mathcal{P}^{\prime}$}
	\IF{$FlagList2[i] = 0$}
		\STATE $\mathcal{P}^{\prime \prime}$ $\leftarrow$ $\mathcal{P}^{\prime \prime} \cup \{ \mathcal{P}^{\prime} [i] \}$
		\FOR{$j=i+1$ \TO $t$}
			\STATE Use Algorithm \ref{alg:isomorphic_tilde} regarding all $t_i$'s as indeterminates 
			\IF{$\texttt{IsIsomorphicN1} ( \mathcal{P}^{\prime} [i], \mathcal{P}^{\prime} [j], q )$ returns \textbf{ISOMORPHIC}}
				\STATE $FlagList2[j]$ $\leftarrow$ $1$
			\ENDIF
		\ENDFOR
	\ENDIF
\ENDFOR
\RETURN $\mathcal{P}^{\prime \prime}$
\end{algorithmic}
\end{algorithm}

\end{document}